\documentclass[reqno]{amsart}
\usepackage{amsmath,amsfonts,amssymb}
\usepackage[inner=1.0in,outer=1.0in,bottom=1.0in, top=1.0in]{geometry}
\newtheorem{theorem}{Theorem}[section]

\newtheorem{proposition}[theorem]{Proposition}
\newtheorem{corollary}[theorem]{Corollary}
\newtheorem{lemma}[theorem]{Lemma}
\newtheorem{remark}[theorem]{Remark}

\newtheorem{example}[theorem]{Example}
\numberwithin{equation}{section}
\usepackage{color}

\newcommand\Sc{\c{S}}
\newcommand{\Id}{\mathrm{Id}}
\newcommand{\ab}{\mathrm{ab}}

\newcommand{\tr}{\mathrm{tr}}

\newcommand{\Q}{\mathbb{Q}}
\newcommand{\Z}{\mathbb{Z}}
\newcommand{\PSL}{\mathrm{PSL}}
\newcommand{\Stab}{\mathrm{Stab}}
\newcommand{\SL}{\mathrm{SL}}
\newcommand{\PGL}{\mathrm{PGL}}
\renewcommand{\qed}{ $\sqcup\!\!\!\!\sqcap$}

\begin{document}
\title{Embedding closed totally geodesic surfaces in Bianchi orbifolds}
\author{Junehyuk Jung and Alan W. Reid}

\address{Department of Mathematics, Texas A\&M University, College Station, TX 77845 USA}

\email{junehyuk@math.tamu.edu}

\address{Department of Mathematics, Rice University, Houston, TX 77005 USA}

\email{alan.reid@rice.edu}

\thanks{The first author is partially supported by Sloan Research Fellowship and by NSF grant DMS-1900993. The second author is partially supported by NSF grant DMS-1812397. The first author wishes to thank Matthew Young for many helpful discussions and Rice University for its hospitality whilst working on this paper. The second author wishes to thank Matthew Stover for useful conversations, and both authors wish to thank Peter Sarnak for countless helpful conversations on topics related to this paper.}

\begin{abstract}

We study embedding of closed totally geodesic hyperbolic 2-orbifolds in the Bianchi orbifolds $\mathbb{H}^3/\PSL(2,\mathcal{O}_d)$. Our main result  shows that there is a constant $c$ such that for $d$ large enough there
are at least $cd$ closed embedded totally geodesic hyperbolic 2-orbifolds. Moreover we provide a list which conjecturally consists of those $d$ for which $\mathbb{H}^3/\PSL(2,\mathcal{O}_d)$ does not contain a closed embedded totally geodesic hyperbolic 2-orbifold.
\end{abstract}

\maketitle
%\tableofcontents

\section{Introduction}
\label{intro}

For a square-free integer $d>0$, let $\mathcal{O}_d$ denote the ring of integers of $\mathbb{Q}(\sqrt{-d})$, and let $\Gamma_d$ be the Bianchi group $\PSL(2,\mathcal{O}_d)$.  The Bianchi groups are arithmetic subgroups of $\PSL(2,\mathbb{C})$,
and their commensurability classes provide the totality of commensurability classes of non-cocompact arithmetic lattices in $\PSL(2,\mathbb{C})$.
The Bianchi groups and Bianchi orbifolds $\Omega_d = \mathbb{H}^3/\Gamma_d$ have been studied repeatedly over the years through their connections to number theory and automorphic forms, as well as in the context of the geometry and topology of $3$-dimensional manifolds (we discuss more of these connections in the context of our results below in \S \ref{context}).

This paper is concerned with splittings of Bianchi groups and certain subgroups of finite index that arise from {\em closed} embedded totally geodesic surfaces. Note that for convenience, throughout this paper, we will use the term surface to mean
a quotient of $\mathbb{H}^2$ by a discrete group of isometries, that may have torsion and may contain orientation-reversing elements.

The study of embedded totally geodesic surfaces in $\Omega_d$ again has a long history (see \S \ref{context}), but perhaps most pertinent to the results of this paper is the result of Frohman and Fine \cite{FF} that says $\Gamma_d$ splits as an HNN-extension over $\PSL(2,\mathbb{Z})$ for all $d\neq 1,3$.
In the case of $d=1$,  $\Gamma_1$ splits as free product with amalgmation, with the amalgamated subgroup also being a copy of  $\PSL(2,\mathbb{Z})$ (see \cite{FF}).
Moreover, in all these cases, this splitting is geometric in the sense that it comes from an embedded totally geodesic copy of $\mathbb{H}^2/\PSL(2,\mathbb{Z})$ in $\Omega_d$. It is known that $\Gamma_3$ has Property FA \cite{Fi} and so does not admit any non-trivial splitting.

In addition, it is also shown in \cite{FF} that for $d\neq 1,2,3,7,11$, $\PSL(2,\mathcal{O}_d)$ splits as a free product with amalgamatation, where the
amalgamating subgroup is {\em isomorphic} to a cocompact Fuchsian group of signature $(0,2,2,3,3)$. This subgroup arises from a $2$-dimensional orbifold embedded in $\Omega_d$ obtained from combining two copies of
$\mathbb{H}^2/\PSL(2,\mathbb{Z})$
and produces an accidental parabolic element. In particular the $2$-orbifold is not totally geodesic.

Our first main result is in some sense a combination of these two results. Note, as is well-known (and we describe in detail in \S \ref{parameter}), $\Omega_d$ contains infinitely many {\em immersed} closed totally geodesic surfaces.

\begin{theorem}\label{main}
Let $d$ be as above. There exists a constant $c>0$ such that for all sufficiently large $d$, the Bianchi orbifold $\Omega_d$ contains at least $cd$ and at most $\ll_\epsilon d^{1+\epsilon} $  embedded closed totally geodesic surfaces.\footnote{Here and elsewhere, $A=O_\phi(B)$, or $A\ll_\phi B$ means that there exists a constant $C(\phi)$ depending only on $\phi$ such that $|A| \leq C(\phi) B$.} When $d$ is a prime $\equiv 3\pmod{4}$, the upper bound can be improved to $\ll d$, and so there are infinitely many $d$ where the lower bound is sharp in the order of magnitude. For $d \in \mathcal{E}_1$, where $\mathcal{E}_1$ is given in the appendix, $\Omega_d$ does not contain any embedded closed totally geodesic surface.
\end{theorem}

In the appendix we describe a larger finite set $\mathcal{E} \supset \mathcal{E}_1$ that conjecturally provides the complete list of those $d$ for which $\Omega_d$ does not contain a closed embedded totally geodesic surface. As we discuss in \S \ref{context}, for some small values of $d$, there are topological obstructions to the existence of closed embedded totally geodesic surfaces in $\Omega_d$.

We now refine the discussion to certain subgroups of $\Gamma_d$.  Recall that a subgroup $\Gamma<\Gamma_d$ is called a {\em congruence subgroup} if there
exists an ideal $I\subset \mathcal{O}_d$ so that $\Gamma$ contains the
{\em principal congruence group}:
$$\Gamma_d(I)=\ker\{\Gamma_d\rightarrow \PSL(2,\mathcal{O}_d/I)\},$$
\noindent where $\PSL(2,\mathcal{O}_d/I) = \SL(2,\mathcal{O}_d/I)/\{\pm \Id\}$.
The largest ideal $I$ for which $\Gamma_d(I)<\Gamma$ is called the {\em level}
of $\Gamma$. A manifold $M$ is called {\em congruence} (resp. {\em principal congruence}) if $M$ is isometric to a manifold $\mathbb{H}^3/\Gamma$ where $\Gamma_d(I) < \Gamma < \Gamma_d$ (resp. $\Gamma=\Gamma_d(I)$)
for some ideal $I$.

As in the case of the Bianchi groups, the existence of embedded totally geodesic surfaces in congruence manifolds has been of some interest (see \S \ref{congruence} for more). However, it has become apparent that understanding embeddings of totally geodesic surfaces in congruence manifolds (for small values of $d$ in particular) is subtle, and depends on many things, amongst which is
the splitting behavior of rational primes in $\mathbb{Q}(\sqrt{-d})$ (see \S \ref{congruence}). To that end our next two results are further examples of this (see \S \ref{parameter}
for the definition of discriminant of a totally geodesic surface).

\begin{theorem}\label{main2}
Assume that $p=\mathcal{P}\overline{\mathcal{P}}$ is a rational prime split to $\mathbb{Z}[i]$ and $I\in \{\mathcal{P}, \overline{\mathcal{P}}\}$. Then
$\mathbb{H}^3/\Gamma_1(I)$ contains no closed embedded totally geodesic surfaces. The same statement holds for the prime ideal $<1+i>$.
\end{theorem}

\begin{theorem}\label{main3}
Let $n \in \mathbb{Z}$, then all immersed totally geodesic surfaces of discriminant $\leq \frac{n}{4}$ in $\mathbb{H}^3/\Gamma_d(n)$ are embedded.
\end{theorem}
%{\color{red} I added closed since we don't seem to deal with straight-lines?}
We note that taking $n$ large enough in Theorem \ref{main3} will yield closed, embedded, orientable surfaces in $\mathbb{H}^3/\Gamma_d(n)$.

\section{Further context and consequences}
\label{context}

In this subsection we discuss in more detail some of the remarks in \S \ref{intro} about totally geodesic surfaces in Bianchi orbifolds and their finite covers.

\subsection{Finiteness}
 For any finite area hyperbolic surface of genus $g$ with $c$ cusps,
 it is easy to see that there are infinitely many simple closed geodesics
 apart from the case when $g=0$ and $c=3$. Counting such geodesics is a good deal more delicate, however, it is known that (see \cite{rivin,MR2415399}) the number of simple closed geodesics of length $<L$ is
\[
\sim L^{6g-6+2b+2c}.
\]

Turning to dimension $3$ the situation is very different, many finite volume hyperbolic $3$-manifolds do not contain any closed {\em immersed} totally geodesic surfaces (see for example \cite[Chapter 9]{MR2}), and when such a manifold does contain closed embedded totally geodesic surfaces, there are at most finitely many (see \cite{Hass} which proves something slightly stronger for a class of surfaces that contain the totally geodesic ones). For convenience we present below
a folklore proof of this for finite volume cusped hyperbolic $3$-manifolds.  Thus a natural follow-up problem
(addressed herein for Bianchi orbifolds) in the case where there are closed embedded totally geodesic surfaces is to count them.

\begin{proof}[Folklore proof]
Let $M=\mathbb{H}^3/\Gamma$ be a cusped finite volume hyperbolic $3$ manifold, for which is normalized so that $\infty$ is the fixed point of a parabolic element of $\Gamma$. Suppose that $M$ contains a closed embedded totally geodesic surface $S$, which upon lifting to $\mathbb{H}^3$ provides a collection of embedded and disjoint hyperbolic planes. Let $H$ be one such plane, and assume that it is bounded by the circle
in $\partial\mathbb{H}^3$ given by:
\[
\mathcal{C}~:~|z-z_0|^2=r^2.
\]
Since $S$ is embedded it follows that $\mathcal{C}\cap \gamma\mathcal{C}=\emptyset$ for all parabolic elements $\gamma$ fixing $\infty$.  Choosing
$\gamma=\begin{pmatrix}1 & x \\ 0 &1\end{pmatrix} \in \Gamma$ for which $|x|$ is minimal (and not zero), it follows that
$r < \frac{|x|}{2}$. Hence,  the image in $M$ of the horoball neighborhood $\{(z,r)\in \mathbb{H}^3~:~r>\frac{|x|}{2}\}$ does not intersect any embedded closed totally geodesic surface.

Now if there were infinitely many closed embedded totally geodesic surfaces in $M$, then Ratner's theorem implies that they must equidistribute as the area $\to \infty$. This is a contradiction to the existence of the forbidden region.
\end{proof}

\subsection{Cuspidal Cohomology Problem}
\label{CCP}

Let $\Gamma$ be a non-cocompact Kleinian group acting on ${\mathbb H}^3$ with
finite co-volume, and set $X={\mathbb H}^3/\Gamma$.
Let $\mathcal{U}(\Gamma)$ denote the subgroup of $\Gamma$ generated by
parabolic elements of $\Gamma$. Note that $\mathcal{U}(\Gamma)$ is
a normal subgroup of $\Gamma$, and we define:
$$V(X)=V(\Gamma) = (\Gamma/\mathcal{U}(\Gamma))^{\ab} \otimes_{\Z} {\Q}.$$
Now $V(X)$ can be identified as a subspace of $H_1(X, \mathbb{Q})$ which defines the (degree $1$) cuspidal (co)homology of $X$ (or $\Gamma$). Alternatively this is the subspace of $H_1(X,\mathbb{Q})$ {\em not} arising from
homology coming from the cusps. We denote the dimension of $V(X)$ by $r(X)$ or $r(\Gamma)$.  Note that if $S\hookrightarrow X$ is a closed embedded orientable surface that is non-separating in $X$, then duality provides a class in $V(X)$.

In the setting of  the Bianchi groups, {\em The Cuspidal Cohomology Problem} posed in the $1980$'s asked
which Bianchi groups have $r(\Gamma_d) = 0$.
Building on a construction of Zimmert \cite{Zim} using embedded non-compact totally geodesic surfaces, work of many people (see for example \cite{GS0}, \cite{GSAnnalen}, \cite{Roh}), finally led to the solution of the
Cuspidal Cohomology Problem by Vogtmann \cite{Vo} who determined the list of all values of $d$ with $r(\Gamma_d)=0$, namely:

\begin{theorem}[Vogtmann]
\label{cuspsummary}
$r(\Gamma_d)=0$ if and only if
$$d\in\{1,2,3,5,6,7,11,15,19,23,31,39,47,71\}.$$\end{theorem}

\noindent We note that all the values of $d$ appearing in Theorem \ref{cuspsummary} appear in the set $\mathcal{E}$ of the appendix. In particular, for those $d\neq 39$ appearing in Theorem \ref{cuspsummary},
Theorem \ref{main} applies to show that $\Omega_d$ does not contain a closed embedded totally geodesic surface.

\subsection{Link complements}
\label{links}
One consequence of Theorem \ref{cuspsummary} is that it provides a list of those $d$ for which $\Omega_d$ {\em could have} a finite cover homeomorphic to a link complement in $S^3$.  Many examples of arithmetic link complements in $S^3$
were known prior to \cite{Vo} (see for example \cite{Ba0}, \cite{Ba1}, \cite{Ha}, and \cite{Th}
to name a few), and in \cite{Ba1} it was shown that for
every $d$ as in Theorem \ref{cuspsummary} there does exist an arithmetic link complement covering $\Omega_d$.  Hence we have
the following result \cite{Ba1}:

\begin{theorem}[Baker]
\label{arithlinks}
$\Omega_d$ is covered by an arithmetic link complement in $S^3$ if and only if
$$d\in\{1,2,3,5,6,7,11,15,19,23,31,39,47,71\}.$$\end{theorem}

\noindent It was shown in \cite{Ha} that for $d=1,2,3,7,11$, there are examples of alternating links whose complements $S^3\setminus L$ cover $\Omega_d$. In particular, in combination with work of Menasco \cite{Men}, such link complements
cannot contain closed embedded totally geodesic surfaces, giving an alternate proof of Theorem \ref{main} in these cases.

\subsection{Congruence subgroups}
\label{congruence}

Of particular interest in connection with the theory of automorphic forms has been the existence of (or lack thereof) non-trivial cuspidal cohomology for the congruence groups $\Gamma_{d,0}(\mathcal{P})$ where $\mathcal{P}$
is a prime ideal of $\mathcal{O}_d$
for those $d$ given in Theorem \ref{cuspsummary}.  Here $\Gamma_{d,0}(\mathcal{P})$ is the congruence subgroup corresponding to the preimage of the Borel subgroup of $\PSL(2,\mathcal{O}_d/\mathcal{P})$.

Much of this interest is motivated by comparison with the classical connections between elliptic curves over $\mathbb{Q}$ and non-trivial cuspidal cohomology classes on congruence subgroups of $\PSL(2,\mathbb{Z})$.
In the setting of the Bianchi groups and the groups $\Gamma_{d,0}(\mathcal{P})$ the non-vanishing of the cuspidal cohomology is conjecturally connected to elliptic curves and other abelian varieties defined over imaginary quadratic fields.  We will not describe this in any detail and refer the reader to \cite{Cr}, \cite{Sen1} and \cite{Sen2}.

Now when $\mathcal{P}=p\mathcal{O}_d$ is an inert prime, then a good deal is known about the existence of non-trivial cuspidal cohomology for the congruence groups $\Gamma_{d,0}(\mathcal{P})$
(see for example \cite{GSAnnalen} and \cite{Roh}), and in particular how embedded totally geodesic surfaces can be used to construct cuspidal cohomology for which $r(\Gamma_{d,0}(\mathcal{P})) = r(\mathcal{P})$ grows with $p$.

On the other hand it is unknown whether there are infinitely many rational primes
$p=\mathcal{P}\overline{\mathcal{P}}$ split to $\mathcal{O}_d$ (in what follows we call such a prime $\mathcal{P}$ a degree one prime) for which $r(\mathcal{P})\neq 0$ or $r(\mathcal{P})=0$.

For example, when $d=1$, the work of \cite{Sen1} (which builds on the work much earlier work of \cite{Cr} and \cite{GHM}) shows that for degree one primes $\mathcal{P}$ with $N(\mathcal{P}) \leq 45000$, $r(\mathcal{P})=0$ for $2061$ of these, and only $260$
have $r(\mathcal{P}) >0$, and of these $177$ have $r(\mathcal{P}) =1$. The smallest degree one prime $\mathcal{P}$ for which $r(\mathcal{P}) >0$ occurs when $N(\mathcal{P})=137$ and in this case $r(\mathcal{P})=1$.
The smallest degree one prime $\mathcal{P}$ for which $r(\mathcal{P}) >1$ occurs when $N(\mathcal{P})=433$ and $r(\mathcal{P})=2$ (see \cite{GHM} and \cite{Sen1}).

It is an immediate consequence of Theorem \ref{main2} that in the cases when $r(\mathcal{P}) >0$, none of the non-trivial classes can be dual to a closed embedded totally geodesic surface.

\section{Preparation}
\label{prep}

\subsection{Parametrization of totally geodesic surfaces}
\label{parameter}
We begin by recalling some preliminary facts regarding immersed totally geodesic surfaces \cite{MR1} and \cite[Chapter 9.6]{MR2}.
If $S\hookrightarrow \Omega_d$ is an immersed totally geodesic surface, then the lifts to $\mathbb{H}^3$ consist of a collection of hyperbolic planes, each of which bounds
a circle (or a straight-line) in $\mathbb{C}\cup \{\infty\} \cong \partial\mathbb{H}^3$. The subgroup of $\Gamma_d$ defined by
\[
\Stab(\mathcal{C},\Gamma_d) = \{\gamma\in \Gamma_d : \gamma \mathcal{C} =  \mathcal{C}\}
\]
is conjugate to an arithmetic group of isometries of $\mathbb{H}^2$, containing the (maximal) arithmetic Fuchsian subgroup (of index at most $2$)
\[
\Stab^+(\mathcal{C},\Gamma_d) = \{\gamma\in \Stab(\mathcal{C},\Gamma_d) :  \gamma~\hbox{preserves components of}~\mathbb{C}\setminus \mathcal{C}\}.
\]
Now the equation of such a circle or straight-line has
the form (see \cite[Chapter 9.6]{MR2})
\[
\mathcal{C}~:~a|z|^2+Bz+\bar{B}\bar{z}+c=0
\]
with $a,c \in \mathbb{Z}$ and $B \in \mathcal{O}_d$, and vice versa. Two such circles (or straight-lines) $\mathcal{C}$ and $\mathcal{C}'$ are said to be {\em equivalent} if there exists $\gamma\in\Gamma_d$ such that $\gamma \mathcal{C}=\mathcal{C}'$.
Define $D=|B|^2-ac$ to be the {\em discriminant} of $\mathcal{C}$. This is preserved by the action of  $\Gamma_d$, and hence equivalent circles have the same discriminant (see \cite{MR1}).  If $\Gamma <\Gamma_d$ is a subgroup of finite index, then a $\Gamma$-equivalence class of circles and straight-lines can be associated to a totally geodesic surface in $\mathbb{H}^3/\Gamma$ and vice versa. Hence
we can also refer to the discriminant of  the associated totally geodesic surface.

When $a\neq0$, $\mathcal{C}$ is a circle centered at $-\overline{B}/a$, with radius $\sqrt{D}/a$.  We will be mainly interested in closed totally geodesic surfaces, and to that end we note:

\begin{lemma}
\label{straight_line}
Suppose that $S$ is a totally geodesic surface immersed in $\Omega_d$ associated to the circle $\mathcal{C}$.
\begin{enumerate}
\item If $S$ is closed, then $\gamma(\mathcal{C})$ is a circle (i.e. not a straight-line) for every $\gamma\in \Gamma_d$.
\item If $S$ is non-compact, then the discriminant $D$ of $S$ has the form $\alpha^2+d\beta^2$ for $\alpha,\beta \in \mathbb{Q}$.
\end{enumerate}
\end{lemma}

\begin{proof} Suppose that $\mathcal{C}$ is a straight-line given by $Bz+\bar{B}\bar{z}+c=0$, so that $D=|B|^2$.
From \cite[Chapter 9.6]{MR2} the invariant quaternion algebra associated to $\Stab^+(\mathcal{C},\Gamma_d)$ has Hilbert Symbol
$$\biggl ({\frac{-d,D}{\mathbb{Q}}}\biggr ) = \biggl ({\frac{-d,|B|^2}{\mathbb{Q}}}\biggr ).$$
We claim that this quaternion algebra is isomorphic to $M(2,\mathbb{Q})$ from which the result follows. To see this, after possibly clearing denominators from $B$, we can assume that the numerator of the Hilbert Symbol
has the form $(-d,b_1^2+db_2^2)$ with $b_1,b_2\in\mathbb{Z}$.

If $b_1=0$,
the quaternion algebra is isomorphic to one with Hilbert Symbol $\biggl ({\frac{-d,d}{\mathbb{Q}}}\biggr )$, and this is isomorphic to $M(2,\mathbb{Q})$. Thus we can assume that $b_1\neq 0$, but then note that
$-dx^2+(b_1^2+db_2^2)y^2=1$ has a solution for $x,y\in\mathbb{Q}$, namely $x=b_2/b_1$, $y=1/b_1$. As is well-known (see \cite[Chapter 2.3]{MR2} for instance) this is equivalent to the quaternion algebra being isomorphic to $M(2,\mathbb{Q})$.
This proves the first part.\\[\baselineskip]
For the second part, by assumption the invariant quaternion algebra associated to $\Stab^+(\mathcal{C},\Gamma_d)$ is isomorphic to $M(2,\mathbb{Q})$. Again using \cite[Chapter 2.3]{MR2}, we deduce that $-dx^2+Dy^2=1$. The required description of $D$ follows from this. \end{proof}

%\begin{remark} Lemma \ref{straight_line}(2) can also be seen geometrically.  This is clear if $\mathcal{C}$ is a straight-line given by $Bz+\bar{B}\bar{z}+c=0$, so that $D=|B|^2$, so assume that $\mathcal{C}$ is a circle
%$a|z|^2+Bz+\bar{B}\bar{z}+c=0$, or equivalently $|z+\overline{B}/a|^2=D/a^2$. If $S$ is non-compact we can find a parabolic element $\Stab(\mathcal{C},\Gamma_d)$ which must fix a complex number $x/y\in\mathbb{Q}(\sqrt{-d})$.

It will also be convenient to have the following interpretation of these circles and the action of $\Gamma_d$.  To each circle or straight-line $\mathcal{C}$ with equation as above, we can associate a Hermitian matrix
\[
A=\begin{pmatrix}a & B \\ \bar{B} & c\end{pmatrix}
\]
and a binary Hermitian form
\[
Q_A(x,y)=\begin{pmatrix} x & y\end{pmatrix}A\begin{pmatrix} \bar{x} \\ \bar{y}\end{pmatrix},
\]
so that
\[
\mathcal{C}~:~ Q_A(z,1)=0.
\]

We can define an action of $\gamma \in \Gamma_d$ on the set of such matrices $A$ as above by:
\[
\gamma^* A \gamma,
\]
where $*$ denotes conjugate-transpose. In terms of the circles, this action sends $\mathcal{C}$ to $\gamma^{-1} \mathcal{C}$.

Note $D=|B|^2-ac=-\det(A)>0$ which we also refer to as the discriminant of the Hermitian form. We say that $A$ or $(a,B,c)$ is {\em primitive}, if no rational integer divides $\gcd(a,B,c)$.

\subsection{Detecting intersection via a trace inequality}
\label{intersection}
The following simple observation regarding the intersection of two circles or straight-lines in $\mathbb{C}\cup \{\infty\}$ plays an important role in our analysis.

\begin{lemma}\label{lem:tr}
Two circles $a|z|^2+Bz+\overline{Bz}+c=0$ and $a'|z|^2+B'z+\overline{B'z}+c'=0$ (straight-lines, if either $a$ or $a'$ is zero) of the same discriminant $D$ intersect nontrivially if and only if
\[
|\mathrm{Tr}(A(A')^{-1})| <2,
\]
where
\[
A = \begin{pmatrix} a & B \\ \bar{B} &c \end{pmatrix},~\hbox{and }~A' = \begin{pmatrix} a' & B' \\ \bar{B'} &c' \end{pmatrix}.
\]
\end{lemma}
\begin{proof}
First assume that $a,a'\neq 0$. Elementary geometry shows that two circles intersect if and only if the distance between the centers of the two circles is greater than the difference of radii and less than the sum of radii. This, together with the descriptions of centers and
radii given above yields.
\[
\left|\frac{\sqrt{D}}{|a|}-\frac{\sqrt{D}}{|a'|}\right| < \left|\frac{B}{a}-\frac{B'}{a'}\right|<\left|\frac{\sqrt{D}}{|a|}+\frac{\sqrt{D}}{|a'|}\right|.
\]
After some manipulation, we see that this is equivalent to
\[
|ac'+ca'-(B'\overline{B}+\overline{B}'B)|<2D.
\]
The statement follows from the trace identity:
\begin{equation}\label{identity}
\mathrm{Tr}(A(A')^{-1}) = -(\frac{ac'+ca'-(B'\overline{B}+\overline{B}'B)}{D}).
\end{equation}
Now assume that $a'=0$, and set $B=B_1+B_2i$, $B'=B_1'+B_2'i$.
Then nontrivial intersection between the circle $a|z|^2+Bz+\bar{B}\bar{z}+c=0$ and the straight-line $B'z+\bar{B'}\bar{z}+c'=2B_1'x-2B_2' y+c'=0$ occurs if and only if the distance from the center of the circle to the line is strictly less than the radius of the circle, i.e.
\[
\frac{|2B_1'\cdot\frac{-B_1}{a}-2B_2'\cdot \frac{B_2}{a}+c'|}{\sqrt{4B_1'^2 + 4B_2'^2}} < \frac{\sqrt{D}}{|a|}.
\]
If $a=a'=0$, then a non-trivial intersection occurs if and only if the two lines are not parallel:
\[
|B_1B_1'+B_2B_2'| < B_1^2+B_2^2.
\]
As above, after some manipulation, we see that these are both equivalent to
\[
|ac'+ca'-(B'\overline{B}+\overline{B}'B)|<2D.
\]
The proof  in both cases is completed using the trace identity \ref{identity}.
\end{proof}

\noindent With this lemma in hand, we can formulate a criterion for an immersed (closed) totally geodesic surface to be embedded.

\begin{corollary}
\label{cor_embed}
Let $\Gamma<\Gamma_d$ be a subgroup of finite index, and  $A$ be a Hermitian matrix associated to an immersed (closed) totally geodesic surface $\mathcal{S} \subset \mathbb{H}^3/\Gamma$. Then $\mathcal{S}$ is embedded in
$\mathbb{H}^3/\Gamma$ if and only if
\begin{equation}\label{cond}
\left|\mathrm{Tr}(\gamma^*A\gamma A^{-1})\right| \geq 2
\end{equation}
for any $\gamma \in \Gamma$.
\end{corollary}

\begin{proof}
Take $A=\gamma^*A\gamma$ and $A' = A$ in Lemma \ref{lem:tr}.
\end{proof}

Since we are mainly interested in closed surfaces, Lemma \ref{straight_line}
shows that in this case we can always assume that $a\neq 0$ for  all circles under consideration.\\[\baselineskip]
\noindent The proof of Theorem \ref{main3}  is now a quick application of Corollary \ref{cor_embed}.

\begin{proof}[Proof of Theorem \ref{main3}] Let $A=\begin{pmatrix} a & B \\ \bar{B} &c \end{pmatrix}$ be of discriminant $D$, set
$\tilde{A}=\begin{pmatrix}c & -B\\ -\bar{B} & a\end{pmatrix}$, so that $\tilde{A} = -D.A^{-1}$.

Then for any $\gamma \in \Gamma(n)$, we have
\[
\mathrm{Tr}\left(\gamma^* A \gamma \tilde{A} \right) \equiv \mathrm{Tr}\left(A  \tilde{A} \right) \pmod{n}.
\]
Note that $\mathrm{Tr}\left(A  \tilde{A} \right) = -2D$, and so putting this together along with $4D\leq n$,  we deduce that
\[
\left|\mathrm{Tr}\left(\gamma^* A \gamma A^{-1} \right) \right| \geq 2.
\]
Hence, $A$ corresponds to an embedded totally geodesic surface by Corollary \ref{cor_embed}.
\end{proof}

\begin{remark} An important feature of the proof of Theorem \ref{main3} is that the operation $*$ preserves $\Gamma(n)$.  In particular, this applies in the case when $n=p$ where $p\in \mathbb{Z}$ is inert to $\mathbb{Q}(\sqrt{-d})$.
On the other hand, this is not the case, for example, when $n$ is replaced by a prime $\mathcal{P} \subset \mathcal{O}_d$ with
$p=\mathcal{P}\overline{\mathcal{P}}$ a rational prime split to $\mathbb{Q}(\sqrt{-d})$.\end{remark}

We also note the following refinement of the trace condition \eqref{cond} to a more useful form
when $\gamma$ is parabolic. This will be used in \S \ref{picard}.

\begin{lemma}\label{lem:para}
For $\gamma = \begin{pmatrix}1-txy & ty^2 \\ -tx^2 & 1+txy\end{pmatrix}$, we have
\begin{equation}\label{para}
\mathrm{Tr}(\gamma^*A\gamma A^{-1}) = 2-\frac{|t|^2 Q_A(-\bar{x},\bar{y})^2 }{D},
\end{equation}
where  $Q_A(x,y) = \begin{pmatrix}x&y\end{pmatrix}A\begin{pmatrix}\bar{x}\\\bar{y}\end{pmatrix}$, and $D$ is the discriminant of $A$.
\end{lemma}
\begin{proof}
We note that for $B$ with $\mathrm{Tr}(B)=0$,
\begin{align*}
\mathrm{Tr} \left((I+B)^* A (I+B) A^{-1}\right) &=\mathrm{Tr} \left((I+B)^*\right)+ \mathrm{Tr} \left((I+B)^* A B A^{-1}\right)  \\
&= 2+ \mathrm{Tr} \left( A B A^{-1}\right)  +  \mathrm{Tr} \left(B^* A B A^{-1}\right) \\
&= 2+\mathrm{Tr} \left(B^* A B A^{-1}\right).
\end{align*}
Consider a matrix $B=v\otimes w = \begin{pmatrix}v_1 w_1 & v_2 w_1 \\ v_1 w_2 & v_2 w_2\end{pmatrix}$, then
\begin{multline*}
\mathrm{Tr} \left(B^* A_1 B A_2\right) = \mathrm{Tr} \left(\begin{pmatrix}|v_1|^2Q_{A_1}(\bar{w_1},\bar{w_2})& \bar{v_1}v_2 Q_{A_1}(\bar{w_1},\bar{w_2})\\ v_1\bar{v_2}Q_{A_1}(\bar{w_1},\bar{w_2})&|v_2|^2Q_{A_1}(\bar{w_1},\bar{w_2})\end{pmatrix}\begin{pmatrix}a_2&B_2\\\bar{B_2}&c_2\end{pmatrix}\right)\\
= Q_{A_1}(\bar{w_1},\bar{w_2}) \left(a_2 |v_1|^2 + \bar{v_1}v_2 \bar{B_2} + v_1\bar{v_2} B_2 + c_2 |v_2|^2 \right) = Q_{A_1}(\bar{w_1},\bar{w_2}) Q_{A_2}(v_1,v_2).
\end{multline*}
 Therefore if we take $P=I+ t\begin{pmatrix}-x\\y\end{pmatrix} \otimes \begin{pmatrix}y\\x\end{pmatrix} $, then we have
\[
\mathrm{Tr} \left(P^* A P A^{-1}\right) = 2+|t|^2 Q_A(-\bar{x},\bar{y}) Q_{A^{-1}}(y,x) = 2+\frac{|t|^2 Q_A(-\bar{x},\bar{y})^2 }{\det A}.
\]
\end{proof}
As a consequence of the Lemma, we bound the discriminant of the embedded totally geodesic surface in terms of the least non-zero integer represented by the corresponding binary hermitian form.
\begin{corollary}\label{para:cor}
With the same notation from Lemma \ref{lem:para}, if $A$ corresponds to an embedded totally geodesic surface in $\Omega_d$, then
\[
4D \leq \min_{x,y\in \mathcal{O}_d~:~Q_A(x,y) \neq 0} Q_A(x,y)^2.
\]
\end{corollary}
\begin{proof}
Because $A$ corresponds to an embedded totally geodesic surface,
\[
\left|\mathrm{Tr} \left(P^* A P A^{-1}\right)\right| \geq 2
\]
for any choice of $P\in \Gamma_d$. If $P$ is a parabolic element, Lemma \ref{lem:para} then implies
\[
2-\frac{|t|^2 Q_A(-\bar{x},\bar{y})^2 }{D} \leq -2
\]
if $Q_A(-\bar{x},\bar{y})\neq 0$. The result then follows by rearranging and substituting.
\end{proof}
\subsection{Orientability}
\label{orientable}
As stated in \S \ref{intro}, we need to allow for the possibility that the surfaces constructed are non-orientable. In the following discussion,  we will assume that $\mathcal{C}$ is a circle (i.e. not a straight-line) associated to a surface $S$ immersed in $\Gamma_d$.

That the surface $S$ is non-orientable happens precisely when $[\Stab(\mathcal{C},\Gamma_d):\Stab^+(\mathcal{C},\Gamma_d)]=2$ which in turn occurs if and only there exists $\gamma\in \Stab(\mathcal{C},\Gamma_d)$ that interchanges the components of  $\mathbb{C}\setminus \mathcal{C}$. In this case, $\gamma$ acts as an orientation-reversing element on the hyperbolic plane spanned by $\mathcal{C}$.

In order to determine whether $[\Stab(\mathcal{C},\Gamma_d):\Stab^+(\mathcal{C},\Gamma_d)]=1$ or $2$, we first characterize elements in $\Stab(\mathcal{C},\Gamma_d)$.
\begin{lemma}\label{fuchsian}
$\gamma = \begin{pmatrix}x& y \\ z& w\end{pmatrix}\in \PSL(2,\mathbb{C})$ fixes $a|z|^2+Bz+\bar{B}\bar{z}+c=0$ if and only if
\begin{equation}\label{orient}
\left\{\begin{array}{rl}
\pm \left(w-\frac{B}{a}z\right)&= \bar{x} + \frac{\bar{B}}{a}\bar{z}\\
\pm D\frac{\bar{z}}{a} &= - B\left(x+\frac{B}{a}z\right)  + ay +Bw\\
\pm 1&= \left|x+\frac{B}{a}z\right|^2 - D\left|\frac{z}{a}\right|^2,
\end{array}\right.
\end{equation}
where $D=|B|^2-ac$.
\end{lemma}
\begin{proof}
Let $A = \begin{pmatrix}a & B \\ \bar{B} & c\end{pmatrix}$ and let $T = \begin{pmatrix}a & B \\ 0 & 1\end{pmatrix}$. Then $\left(T^{-1}\right)^*A T^{-1}$ corresponds to a circle of radius $D$ centered at origin, and one can check that the set of elements in $\PSL(2,\mathbb{C})$ fixing it is given by
\[
F=\left\{\begin{pmatrix}\alpha & D\beta \\ \pm \bar{\beta} & \pm \bar{\alpha}\end{pmatrix}\in \PSL(2,\mathbb{C})\right\}.
\]
Now note that $\gamma$ fixes $A$ if and only if $T\gamma T^{-1}$ fixes $\left(T^{-1}\right)^*A T^{-1}$, which happens if and only if $T\gamma T^{-1} \in F$. Because
\[
T\gamma T^{-1} = \begin{pmatrix}x+\frac{B}{a}z &- B\left(x+\frac{B}{a}z\right)  + ay +Bw\\ \frac{z}{a} & w-\frac{B}{a}z \end{pmatrix},
\]
we see that \eqref{orient} is a necessary and sufficient condition of $T\gamma T^{-1}$ being an element of $F$.
\end{proof}
Note that $[\Stab(\mathcal{C},\Gamma_d):\Stab^+(\mathcal{C},\Gamma_d)]=2$ if and only if one can find an element in $\Stab(\mathcal{C},\Gamma_d)\backslash \Stab^+(\mathcal{C},\Gamma_d)$. We state this as follows:
\begin{lemma}\label{lem:orient}
A closed totally geodesic surface in $\mathbb{H}^3/\Gamma$
corresponding to $a|z|^2+Bz+\bar{B}\bar{z}+c=0$ is non-orientable, if and only if there exists
$\gamma \in \Gamma$ that solves \eqref{orient} with $\pm$ replaced by $-$.
\end{lemma}

\begin{example} It is straightforward to construct many circles for which $\Stab(\mathcal{C},\Gamma_d)$ contains orientation-reversing elements. For example, let $k>0$ be an integer, and for any $d$ the element
\[
T_k = \begin{pmatrix}1+k\sqrt{-d} & (dk^2+2) \\ -1& -(1-k\sqrt{-d}) \end{pmatrix}
\]
acts as an orientation-reversing element on the hyperbolic plane spanning the circle centered at the origin and discriminant $D=(dk^2+2)$. Note that in this case $\tr(T_k)=2k\sqrt{-d}$ is imaginary but $\tr(T_k^2)=-4k^2d-2\in \mathbb{R}$, and this easliy establishes that $T_k$ acts orientation reversing.

One still needs to arrange for the corresponding surface to be closed, but this can also be arranged; e.g. $d=1$, $k=2$ gives $D=6$ and the quaternion algebra $\biggl ({\frac{-1,6}{\mathbb{Q}}}\biggr )$ is ramified at $2$ and $3$, hence the
surface group is cocompact.\end{example}

\section{Number theoretic lemmas}
\subsection{On quadratic residues}
Before proving our results about embedding totally geodesic surfaces, we prove some number theoretic lemmas that will be used in what follows. We first recall an estimate concerning small prime quadratic non-residues modulo $d$ from \cite{BouLin}.
\begin{theorem}\cite[{Theorem 5.1}]{BouLin}\label{lem:bou}
For any $\epsilon>0$ there is $\alpha>0$ so that for every large enough integer $D$ which is not a perfect square, and $N\geq |D|^{1/4+\epsilon}$ one has that the set $P$ of primes $N^\alpha \leq p\leq N$ with $\left(\frac{D}{p}\right)=-1$ satisfy
\[
\sum_{p\in P} \frac{1}{p} > \frac{1}{2}-\epsilon.
\]
\end{theorem}
We will need a variation of Theorem \ref{lem:bou}.
\begin{corollary}\label{crit}
There exist constants $\alpha, \beta>0$ such that for all sufficiently large square free $d$, there is a collection $P_d$ of primes  $d^\alpha \leq p \leq d^{1/4-\beta}$ with $\left(\frac{-d}{p}\right)=-1$ so that
\[
\frac{3}{4}> \sum_{p\in P_d} \frac{1}{p} > \frac{1}{4}.
\]
\end{corollary}
\begin{proof}
For a given $\epsilon>0$, let $N=d^{1/4+\epsilon}$ and let $\alpha(\epsilon)$ be the constant implied by Theorem \ref{lem:bou}. Then we have
\[
\sum_{p\in P} \frac{1}{p} - \sum_{\substack{p\in P \\ p<d^{1/4-\epsilon}}} \frac{1}{p} < \log \log d^{1/4+\epsilon} - \log \log d^{1/4-\epsilon}+o(1) = \log \frac{1/4+\epsilon}{1/4-\epsilon}+o(1) = O(\epsilon),
\]
where we used Mertens' second theorem in the first inequality and $o(1)$ is absorbed into $O(\epsilon)$ by taking $d$ sufficiently large. Therefore we may take $\epsilon>0$ sufficiently small (yet fixed) so that
\[
\sum_{\substack{p\in P \\ p<d^{1/4-\epsilon}}} \frac{1}{p} > \frac{1}{4}
\]
for all sufficiently large $d$. We let $\beta$ be this choice of $\epsilon$ and let $\alpha=\alpha(\beta)$. Now let $p_1<p_2 < \ldots < p_n <d^{1/4-\beta}$ be the complete set of primes in $\{p\in P ~:~ p<d^{1/4-\beta}\}$. Then the partial sums
\[
S_k=\sum_{i=1}^k \frac{1}{p_i}
\]
satisfy
\[
S_{k+1}-S_k = \frac{1}{p_{k+1}} < d^{-\alpha}
\]
which we may assume to be less than $\frac{1}{4}$ under the assumption of $d$ being sufficiently large. Also we have $S_1<1/2$ and $S_n > 1/4$. Therefore there exists $1<  k\leq n$ such that $1/4<S_k <3/4$. We take
\[
P_D = \{p_1, \ldots, p_k\},
\]
to complete the proof.
\end{proof}
The main purpose of this section is to estimate a weighted summation over the following set
\begin{equation}\label{D}
\mathcal{D}=\{0<r<\frac{d}{4}~:~ r \text{ is a square modulo }d,~r\neq \alpha^2 + d\beta^2~\forall \alpha,\beta \in \mathbb{Q} \},
\end{equation}
where $d$ is a square free integer. This is because, every embedded totally geodesic surface we construct has a discriminant which is a square modulo $d$, while the discriminant not being of the form $\alpha^2 + d\beta^2 $ with $\alpha,\beta \in \mathbb{Q}$ is a necessary and sufficient condition for the surface being closed (Lemma \ref{straight_line}). The main theorem of the section is the following. We use the notation $\omega(n)$ to be the number of distinct prime divisors of $n$:
\begin{theorem}\label{lem2}
For all sufficiently large $d$, we have
\[
\sum_{r\in \mathcal{D}} 2^{\omega(d)-\omega(\gcd(d,r))} \gg d.
\]
\end{theorem}
\begin{proof}
To estimate the first summation, we set $f(r)=2^{\omega(d)-\omega(\gcd(d,r))}$ if $r$ is a square modulo $d$ and $0$ otherwise. One can check that $f$ has the following closed form:
\[
f(r)=\prod_{p|d,~p:\text{ odd}} \left(\left(\frac{r}{p}\right)+1\right),
\]
where $\left(\frac{\cdot}{\cdot}\right)$ is the Jacobi symbol. Now we set $\mathrm{P}$ to be the set of odd primes $p$ such that
\[
\left(\frac{-d}{p}\right) = -1.
\]
Note that Legendre's theorem implies that $r$ is of the form $\alpha^2+d\beta^2$ if and only if
\begin{itemize}
\item $r$ is a square modulo $d$, and
\item $-d$ is a square modulo the squarefree part of $r$.
\end{itemize}
Therefore we have
\begin{equation}\label{g:lower}
\sum_{r\in \mathcal{D}} 2^{\omega(d)-\omega((d,r))} \geq \sum_{\substack{0<r<d/4\\ \text{the square free part of }r \text{ has a factor in }\mathrm{P}}} f(r).
\end{equation}
To estimate the summation on the right hand side, we first let $q$ be a squarefree integer whose prime factors are all in $\mathrm{P}$, and let $r=qm$ with $\gcd(m,q)=1$. Then
\[
\sum_{\substack{0<r=mq < \frac{d}{4}\\ \gcd(m,q)=1}} f(r) =  \sum_{\substack{0<m < \frac{d}{4q}\\ \gcd(m,q)=1}} \sum_{e|d,~\gcd(e,2)=1} \left(\frac{q}{e}\right)\left(\frac{m}{e}\right)
 = \left\lfloor\frac{d}{4q}\right\rfloor  +  \sum_{\substack{e|d,~e>1\\ \gcd(e,2)=1}} \left(\frac{q}{e}\right) \sum_{\substack{0<m < \frac{d}{4q}\\ \gcd(m,q)=1}}\left(\frac{m}{e}\right).
\]
Inputting the Polya-Vinogradov inequality \cite{dav},
\[
\left|\sum_{r=N}^{N+M} \chi(r)\right| <\sqrt{q_\chi}\log q_\chi,
\]
where $q_\chi$ is the conductor of $\chi$, we see that for $e>1$,
\begin{multline*}
\left|\sum_{\substack{0<m < \frac{d}{4q}\\ \gcd(m,q)=1}}\left(\frac{m}{e}\right)\right| =\left| \sum_{f|q}\mu(f)\sum_{0<mf < \frac{d}{4q}}\left(\frac{mf}{e}\right)\right|=\left| \sum_{f|q}\mu(f)\left(\frac{f}{e}\right)\sum_{0<m < \frac{d}{4fq}}\left(\frac{m}{e}\right)\right| \\
\leq \sum_{f|q}\left|\sum_{0<m < \frac{d}{4fq}}\left(\frac{m}{e}\right)\right| < \tau(q)\sqrt{e}\log e \leq \tau(q) \sqrt{e}\log d,
\end{multline*}
\footnote{We let $\tau_x(n) = \sum_{d|n}d^x$, and for $x=0$, we may omit $0$ and write $\tau=\tau_0$ instead.}
and so
\begin{equation}\label{eqeq}
\left|\sum_{\substack{0<r=mq < \frac{d}{4}\\ (m,q)=1}} f(r)-\frac{d}{4q}\right|< 1+ \tau(q)\tau_{\frac{1}{2}}(d) \log d.
\end{equation}
%The second assertion of the lemma follows from bounding \eqref{sum} from below by $\sum_{\substack{0<r=mq < \frac{d}{4}\\ (m,q)=1}} f(r)$.

Now observe that the right hand side of \eqref{g:lower} is bounded from below by the following sum
\begin{equation}\label{final}
S=\sum_{p\in P_d} \sum_{\substack{0<r=mp < \frac{d}{4}\\ \gcd(m,p)=1}} f(r) - \sum_{p_1\neq p_2\in P_d} \sum_{\substack{0<r=mp_1p_2 < \frac{d}{4}\\ \gcd(m,p_1p_2)=1}} f(r)
\end{equation}
by the inclusion-exclusion principle, where $P_d\subset \mathrm{P}$ is the set of primes given in Corollary \ref{crit}. Because all primes in $P_d$ is less than $d^{1/4-\beta}$, we see that $p_1p_2 < d^{1/2-2\beta}$. Hence by assuming that $d$ is sufficiently large, we may assume for any given $\varepsilon>0$ using \eqref{eqeq} that
\[
\sum_{\substack{0<r=mp < \frac{d}{4}\\ \gcd(m,p)=1}} f(r) > (1-\varepsilon)\frac{d}{4p}
\]
for $p\in P_d$, and that
\[
\sum_{\substack{0<r=mp_1p_2 < \frac{d}{4}\\ \gcd(m,p_1p_2)=1}} f(r) <(1+\varepsilon) \frac{d}{4p_1p_2}
\]
for $p_1,p_2 \in P_d$. In particular, we have
\[
S > \sum_{p\in P_d}(1-\varepsilon)\frac{d}{4p}  - \sum_{p_1\neq p_2\in P_d}(1+\varepsilon) \frac{d}{4p_1p_2}  > \frac{d}{4}  \left((1-\varepsilon) \sum_{p\in P_d} \frac{1}{p} - (1+\varepsilon) \left(\sum_{p\in P_d} \frac{1}{p}\right)^2\right),
\]
which we see that is $\gg d$ by taking $\varepsilon$ sufficiently small and using
\[
1/4<\sum_{p\in P_d} \frac{1}{p} < 3/4
\]
by Corollary \ref{crit}.
\end{proof}
\subsection{On binary hermitian forms}
To prove an upper bound of the number of embedded totally geodesic surfaces, we need to understand the least integer represented by a binary hermitian form, in the spirit of Corollary \ref{para:cor}. To this end, we show in this section that every primitive binary hermitian form over
$\mathcal{O}_d$ of discriminant $D$ represents a positive integer $\ll \gcd(d,D) \log^2 \frac{D}{\gcd(d,D)}$. Firstly, we recall the local-global principle for indefinite integral forms.
\begin{lemma}{\cite[p. 131]{MR522835}}
Let $q$ be a regular indefinite integral form in $n\geq 4$ variables, and let $a\neq 0$ be an integer. Suppose that $a$ is represented by $q$ over all $\mathbb{Z}_p$. Then $a$ is represented by $q$ over $\mathbb{Z}$.
\end{lemma}
We also need the following classical fact about solving an equation in $\mathbb{Z}_p$.
\begin{proposition}\label{prop}
Let $f(x)$ be a polynomial with integral coefficients. One can solve $f(x)=0$ in $\mathbb{Z}_p$ if there is an integer $x_0$ such that
\[
f(x_0) \equiv 0 \pmod{p}
\]
and
\[
f'(x_0) \not\equiv 0 \pmod{p}.
\]
\end{proposition}

\begin{theorem}\label{upper}
Let $a,c \in \mathbb{Z}$ and let $B\in \mathcal{O}_d$. Let $Q(x,y) = a|x|^2 + 2\mathrm{Re}(Bx\bar{y})+c|y|^2$ be a primitive binary hermitian form, i.e., no positive integer other than $1$ divides all $a,B,c$. Let $D=|B|^2-ac$. Then $Q(x,y)$ represents a positive integer whose modulus is $O(\gcd(d,D) \log^2 (D/\gcd(d,D)))$.
\end{theorem}
\begin{proof}
Let $B=B_1+\sqrt{-d}B_2$. Then the binary hermitian form $Q$ over $\mathcal{O}_d$ is a quaternary quadratic form over $\mathbb{Z}$:
\[
q(x_1,x_2,y_1,y_2) = ax_1^2+adx_2^2+2B_1(x_1y_1+dx_2y_2)+ 2dB_2 ( x_2y_1 - x_1y_2) + cy_1^2+cdy_2^2,
\]
whose discriminant is $16d^2D^2$.

Let $(a,c,2B_1,2dB_2)=e=2^k e_0$, where $e_0$ is odd, and we note here that $k\leq 2$. Let $a'=a/e,c'=c/e,~B_1'=B_1/e,~d'=d/e$. Then
\[
\frac{q}{e}=q'(x_1,x_2,y_1,y_2) = a'x_1^2+a'dx_2^2+2B_1'(x_1y_1+dx_2y_2)+ 2d'B_2 ( x_2y_1 - x_1y_2) + c'y_1^2+c'dy_2^2,
\]
is primitive. We claim that $q'$ represents an integer whose modulus is $O\left(\frac{\gcd(d,D)}{e_0} \log^2 (D/\gcd(d,D))\right) $.

We first note for an odd prime $p| \gcd(d,D)/e_0$ that $q'$ represents at least one rational integer $\alpha_p$ in $\{1,2,\ldots,p\}$ over $\mathbb{Z}_p$. For $p=2$, observe that a primitive quadratic form represents at least one rational integer $\alpha_2$ in $\{1,2,\ldots,8\}$ over $\mathbb{Z}_2$. From this %%AR: changed these to this
data we choose $0<\alpha \leq 8\gcd(d,D)/e_0$ to be the integer such that $\alpha\equiv \alpha_p \pmod{p}$ for odd primes $p$ and $\alpha\equiv \alpha_2 \pmod{8}$ using Chinese remainder theorem.

Observe that if $p$ does not divide $dD$, then $q'$ represents any rational integer over $\mathbb{Z}_p$. We now consider three cases:\\[\baselineskip]

\noindent
\textbf{Case 1:}~If $p|e$, then $p\nmid B_2$, because $q$ is primitive. So if we set
\[
f(x)=q'(0,x,1,0) = a'dx^2+2d'B_2 x + c',
\]
we have $f(x) \equiv 2d'B_2 x + c' \pmod{p}$ which represents any number mod $p$, and $f'(x) \equiv 2d'B_2 \pmod{p}$ is nonzero. Therefore $f(x)$ represents any rational integer over $\mathbb{Z}_p$ by
Proposition \ref{prop}.\\[\baselineskip]
\textbf{Case 2:}~If $p|d$, but $p \nmid D$, then
\[
q'(x,0,y,0) = a'x^2 + 2B_1' xy + c'y^2
\]
which is a binary quadratic form with the discriminant that is congruent to $4D/e$ modulo $p$. So $q'(x,0,y,0)$ represents any rational integer of the form $p^{2k}m$ with $p\nmid m$, over $\mathbb{Z}_p$. Likewise, because $p\| d$,
\[
q'(0,x,0,y) = d(a'x^2+2B_1'xy+c'y^2)
\]
represents any rational integer of the form $p^{2k+1}m$ with $p\nmid m$, over $\mathbb{Z}_p$.\\[\baselineskip]
\textbf{Case 3:} If $p|D=B_1^2+dB_2^2-ac$, we consider two scenarios: when $p|a$ and $p|c$, and when $p\nmid a$ or $p \nmid c$. In the first scenario, $p\nmid B_1$, and so
\[
f(x)=q'(x,0,1,0) = a'x^2 + 2B_1' x + c'
\]
represents any rational integer over $\mathbb{Z}_p$, by Proposition \ref{prop}.

In the second scenario, we assume without loss of generality that $p \nmid a$. Then
\[
q'(x,y,0,0) = a'(x^2+ dy^2),
\]
and because $p\nmid d$, this represents any rational integer coprime to $p$.

Combining all these, we see that $q'$ represents
\begin{equation}\label{repre}
\frac{8\gcd(d,D)}{e_0}N + \alpha
\end{equation}
if it is coprime to the odd part of $D/\gcd(d,D)$, which we denote by $M$. Let $\beta$ be the multiplicative inverse of $\frac{8\gcd(d,D)}{e}$ modulo $M$. Hence we want to find $N$ such that $N+\alpha\beta$ is coprime to $M$.

We recall that the Jacobsthal function $j(n)$ is the smallest number $m$ such that every sequence of $m$ consecutive integers contains an integer coprime to $n$, and that we have $j(n) = O(\log ^2 n)$ \cite{MR499895}.

So one can choose $N  \ll \log^2 M \ll \log^2 (D/\gcd(d,D))$ to make \eqref{repre} represented by $q'$. In other words, $q'$ represents an integer
\[
\ll \frac{8\gcd(d,D)}{e_0} \log^2 (D/\gcd(d,D))
\]
and so $q$ represents an integer
\[
\ll e\frac{8\gcd(d,D)}{e_0} \log^2 (D/\gcd(d,D))   \ll \gcd(d,D) \log^2 (D/\gcd(d,D)).
\]
where we used the fact that $e/e_0 = 2^k \leq 4$.
\end{proof}
Finally, we recall a general bound for the number of immersed totally geodesic surfaces having a given discriminant.
\begin{lemma}\label{upper2}
For each fixed $D$, there are at most
\[
4\tau(\gcd(d,D))^2=4\prod_{p|\gcd(d,D)} 4
\]
inequivalent binary hermitian forms of the discriminant $D$.
\end{lemma}
\begin{proof}
This follows from Theorem 5.1, 5.3, 5.5, and 5.7 of \cite{JM}.
\end{proof}
\section{The Bianchi groups I: Existence and lower bounds}
\label{bianchi}
It will be convenient to recall some additional notation. We let $B_d$ denote the maximal discrete group containing $\Gamma_d$, which following \cite{Vul} we refer to as the {\em extended Bianchi group}.
The group $B_d$ contains $\PGL(2,\mathcal{O}_d)$ as a normal subgroup, and we describe $B_d$ by adjoining elements to $\PGL(2,\mathcal{O}_d)$. (Note that $\PGL(2,\mathcal{O}_d)$ contains $\Gamma_d$ with index $2$.) Let
\[
\delta = \left\{\begin{array}{cr} -d &\text{if } d\equiv 3 \pmod{4}\\-4d &\text{if } d\not\equiv 3 \pmod{4} \end{array}\right.
\]
denote the discriminant of $\mathbb{Q}(\sqrt{-d})$. For a positive square-free integer $r|\delta$, let $\sigma_r$ is a matrix with entries in $\mathcal{O}_d$ with $\det \sigma_r = r$.  Define
\[
\tilde{\sigma}_r = \sigma_r \begin{pmatrix}\frac{1}{\sqrt{r}} & 0 \\ 0 & \frac{1}{\sqrt{r}}\end{pmatrix} \in \PSL(2,\mathbb{C}).
\]
Then we have the decomposition:
\[
B_d = \coprod_{\substack{r|\delta,~0<r< \sqrt{|\delta|}\\r:\text{ square-free}}}\tilde{\sigma}_r \PGL(2,\mathcal{O}_d),
\]
and in particular, we have $\lbrack B_d: \PGL(2,\mathcal{O}_d) \rbrack = 2^{\omega(\delta)-1}$.
\begin{remark}
It follows that when $d=1,2$ or is a prime congruent to $3$ modulo $4$ then $B_d=\PGL(2,\mathcal{O}_d)$.
\end{remark}

For $B\in\mathcal{O}_d$ and $c\in \mathbb{Z}$, let $\mathcal{S}_{B,c}$ be the immersed totally geodesic surface in $\Omega_d$
whose associated circle in $\partial\mathbb{H}^3$ is given by
\[
\mathcal{C}_{B,c}~:~d|z|^2 + B\sqrt{-d}z +\overline{B\sqrt{-d}}\overline{z}+dc=0.
\]
In this section, we analyze totally geodesic surfaces associated to the circles obtained as
\begin{equation}\label{rep}
\tilde{\sigma}_r^{-1} \mathcal{C}_{B,c}
\end{equation}
for distinct $B \in \mathcal{O}_d$ and $c\in \mathbb{Z}$. Note that the circles $\sigma_r^{-1} \mathcal{C}_{B,c}$ are indeed associated to totally geodesic surfaces in $\Omega_d$ since $\mathcal{C}_{B,c}$ %%AR: changed S to C
has equation
of the form in \S \ref{parameter}, and $\sigma_r^{-1}\in B_d$ which contains $\Gamma_d$ as a subgroup of finite index.

Note that %%AR: added that
from \cite[Theorem 8]{Vul}  if the ideal class group of the imaginary quadratic field $\mathbb{Q}(\sqrt{-d})$ does not contain any element of order $4$, then every immersed totally geodesic surface in $\Omega_d$ is associated to a circle of the form \eqref{rep} for some $|B|<d$ and $c\in \mathbb{Z}$. For other $d$, such a parametrization is not known. Nevertheless, \eqref{rep} in general forms a rich subfamily of totally geodesic surfaces, and we investigate embedded closed totally geodesic surfaces of the form \eqref{rep} in subsequent sections.

\subsection{Lifting immersed totally geodesic surfaces}
\begin{proposition}\label{prop1}

Fix $m, c\in \mathbb{Z}$. %%AR: tidied up the first line
Then among the circles
\[
\tilde{\sigma}_r^{-1} \mathcal{C}_{m,c} \subset \partial \mathbb{H}^3
\]
with $r|\delta$ and $0<r<\sqrt{|\delta|}$, one can choose at least $2^{\omega(\delta)-1}/\tau(8\gcd(d,m))$ circles such that none of them is $\PGL(2,\mathcal{O}_d)$-equivalent to the other.
\end{proposition}
Before proving the proposition, we recall that the action of $\PGL(2,\mathbb{C})\cong \PSL(2,\mathbb{C})$ on $\partial\mathbb{H}^3$ via fractional transformation is well-defined, and can be extended
to an action on $\mathbb{H}^3 = \{z+rj~:~ z\in \mathbb{C}, r>0\}$. %%AR: rewrote as I didnt understand " it is necessary to normalize $\gamma$ to have determinant one (or positive, at least)."
\begin{proof}
Recall that $\PGL(2,\mathcal{O}_d)$ is a normal subgroup of $B_d$, with
\[
\{\tilde{\sigma}_r~:~ r|\delta,~ 0< r< \sqrt{|\delta|}\}
\]
being a complete coset representatives. Hence $\PGL(2,\mathcal{O}_d)\tilde{\sigma}_{r_1}^{-1} \mathcal{C}_{m,c}$ and $\PGL(2,\mathcal{O}_d)\tilde{\sigma}_{r_2}^{-1} \mathcal{C}_{m,c}$ are identical if and only if $\Stab(\mathcal{C}_{m,c},B_d)$ intersects
\[
\tilde{\sigma}_{r_1}\PGL(2,\mathcal{O}_d)\tilde{\sigma}_{r_2}^{-1}=\tilde{\sigma}_{r_1}\tilde{\sigma}_{r_2}^{-1}\PGL(2,\mathcal{O}_d)
\]
non-trivially. In order to prove the claim, therefore, it is sufficient to prove that there are at most $\tau(8\gcd(d,m))$ distinct $r$ with $r|\delta$ $~ 0< r< \sqrt{|\delta|}$ such that
\[
\tilde{\sigma}_{r}\PGL(2,\mathcal{O}_d)
\]
has non-trivial intersection with $\Stab(\mathcal{C}_{m,c},B_d)$.

We first specify the choice of $\sigma_r$. When $d \not\equiv 1 \pmod {4}$, we let
\[
\sigma_r = \begin{pmatrix}\sqrt{-d} & r \\ vr & u\sqrt{-d}\end{pmatrix},
\]
with $-d=rs$, $us-vr=1$. When $d\equiv 1 \pmod{4}$ and $r$ is odd, we define $\sigma_r$ the same as above, and when $r=2r_0$ is even, we let
\[
\sigma_r =\begin{pmatrix}1+\sqrt{-d} & 2 \\ 2v_0 & u_0(1+\sqrt{-d})\end{pmatrix} \sigma_{r_0}
\]
where $-1-d=2s_1$ and $u_0s_1-2v_0=1$.

We then have
\begin{multline*}
\tilde{\sigma}_r \PGL(2,\mathcal{O}_d)= \{\begin{pmatrix}x&y\\z&w\end{pmatrix}=\begin{pmatrix}x_1\sqrt{r}+x_2\sqrt{-s} & y_1\sqrt{r}+y_2\sqrt{-s}\\ z_1\sqrt{r}+z_2\sqrt{-s}& w_1\sqrt{r}+w_2\sqrt{-s}\end{pmatrix}\in \PGL(2,\mathbb{C})\\
:~x_1+x_2\sqrt{-d},y_1+y_2\sqrt{-d},z_1+z_2\sqrt{-d},w_1+w_2\sqrt{-d}\in \mathcal{O}_d\},
\end{multline*}
if $d \not\equiv 1 \pmod {4}$, or if $d\equiv 1 \pmod{4}$ and $r$ is odd. When $d\equiv 1 \pmod{4}$ and $r$ is even, we write
\[
\tilde{\sigma}_r \PGL(2,\mathcal{O}_d)
= \{\begin{pmatrix}\frac{\sqrt{2}+\sqrt{-2d}}{2} & \sqrt{2} \\ \sqrt{2}v_0 & u_0\frac{\sqrt{2}+\sqrt{-2d}}{2}\end{pmatrix}\begin{pmatrix}x&y\\z&w\end{pmatrix}~:~ \begin{pmatrix}x&y\\z&w\end{pmatrix} \in \tilde{\sigma}_{r_0}\PGL(2,\mathcal{O}_d)\}.
\]

Note that Proposition \ref{prop1} %%AR: reworded
is trivial when $d=1,3$. So we assume that $d\neq 1,3$, which makes
\[
\PGL(2,\mathcal{O}_d) = \{\gamma \in M_{2\times 2}(\mathcal{O}_d) ~:~ \det \gamma = \pm 1\}.
\]
Observe in both cases that if we write $\delta = r q$, then
\[
\tilde{\sigma}_r \PGL(2,\mathcal{O}_d) \subset  \{\begin{pmatrix}x&y\\z&w\end{pmatrix}=\begin{pmatrix}x_1\sqrt{r}+x_2\sqrt{-q} & y_1\sqrt{r}+y_2\sqrt{-q}\\ z_1\sqrt{r}+z_2\sqrt{-q}& w_1\sqrt{r}+w_2\sqrt{-q}\end{pmatrix} ~:~ x_i,y_i,z_i,w_i \in \frac{1}{2}\mathbb{Z},~ xw-yz=\pm 1 \}.
\]
%%AR: removed part of a line Considering their action to $\partial \mathbb{H}^3$,
and we may therefore regard $\tilde{\sigma}_r \PGL(2,\mathcal{O}_d)$ as a subset of
\[
\{\begin{pmatrix}x&y\\z&w\end{pmatrix}=\begin{pmatrix}x_1\sqrt{\pm r}+x_2\sqrt{\mp q} & y_1\sqrt{\pm r}+y_2\sqrt{\mp q}\\ z_1\sqrt{\pm r}+z_2\sqrt{\mp q}& w_1\sqrt{\pm r}+w_2\sqrt{\mp q}\end{pmatrix} ~:~ x_i,y_i,z_i,w_i \in \frac{1}{2}\mathbb{Z},~xw-yz=1 \}  \subset \PSL(2,\mathbb{C}).
\]
Recall from Lemma \ref{fuchsian} with $a=d$, $B= m\sqrt{-d}$, and $c=cd$
that $\beta = \begin{pmatrix}x&y\\z&w\end{pmatrix} \in \PSL(2,\mathbb{C}) $
fixes $\mathcal{C}_{B,c}$  if and only if
\begin{equation}\label{equi}
\left\{\begin{array}{rl}
\pm1&=|x|^2+2\mathrm{Re}\left(\frac{\sqrt{-d}}{d}mz\bar{x}\right)+c|z|^2\\
w&=\frac{m\sqrt{-d}}{d} (z\mp\bar{z})\pm\bar{x}\\
y&=\frac{m\sqrt{-d}}{d} (x\mp\bar{x})\pm c\bar{z}.
\end{array}\right.
\end{equation}
We claim that if
\[
\begin{pmatrix}x&y\\z&w\end{pmatrix}=\begin{pmatrix}x_1\sqrt{\pm r}+x_2\sqrt{\mp q} & y_1\sqrt{\pm r}+y_2\sqrt{\mp q}\\ z_1\sqrt{\pm r}+z_2\sqrt{\mp q}& w_1\sqrt{\pm r}+w_2\sqrt{\mp q}\end{pmatrix},
\]
with  $x_i,y_i,z_i,w_i \in \frac{1}{2}\mathbb{Z}$ and $xw-yz=1$ satisfies \eqref{equi}, then it is necessary that either the odd part of $r$ or the odd part of $q$ must divide $m$.

The first equation in \eqref{equi} is equivalent to
\begin{equation}\label{first}
\pm 1= rx_1^2 + qx_2^2 +2m (z_1x_2 - z_2 x_1) + cr z_1^2 + cq z_2^2.
\end{equation}

We first deal with the following case:
\[
\begin{pmatrix}x&y\\z&w\end{pmatrix}=\begin{pmatrix}x_1\sqrt{r}+x_2\sqrt{-q} & y_1\sqrt{r}+y_2\sqrt{-q}\\ z_1\sqrt{r}+z_2\sqrt{-q}& w_1\sqrt{r}+w_2\sqrt{-q}\end{pmatrix}.
\]
We consider $+1$ and $-1$ cases of \eqref{equi} separately:\\[\baselineskip]
\textbf{Case $+1$:} In this case, if $r$ has an odd prime factor $p$ that does not divide $m$, then the second and the third condition of \eqref{equi} imply that $p|z_2$ and $p|x_2$, and so there are no solutions to the \eqref{first}, because it would imply
\[
+1 \equiv 0 \pmod{p}.
\]
Hence the system of equations can have a solution only if the odd part of $r$ divides $m$.\\[\baselineskip]
\textbf{Case $-1$:} Likewise, if $q$ has an odd prime factor $p$ that does not divide $m$, then the same argument shows that there are no solutions to \eqref{first}. So the system of equations may have a solution only
if the odd part of $q$ divides $m$.\\[\baselineskip]
Applying the same argument to the case
\[
\begin{pmatrix}x&y\\z&w\end{pmatrix}=\begin{pmatrix}x_1\sqrt{-r}+x_2\sqrt{q} & y_1\sqrt{-r}+y_2\sqrt{q}\\ z_1\sqrt{-r}+z_2\sqrt{q}& w_1\sqrt{-r}+w_2\sqrt{q}\end{pmatrix},
\]
we conclude that $\tilde{\sigma}_r \PGL(2,\mathcal{O}_d)$ intersects $\Stab(\mathcal{C}_{m,c},B_d)$ %%AR: changed C to mathcal{C}
nontrivially only if the odd part of $r$ or the odd part of $\delta/r$ divides $m$. Hence only if $r$ or $\delta/r$ divides $8\gcd(d,m)$. Among $0<r<\sqrt{|\delta|}$, there are at most $\tau (8\gcd (d,m))$ of them, concluding the proof.
\end{proof}
\begin{corollary}\label{duplicate}
Among the totally geodesic surfaces $\tilde{\sigma}_r \mathcal{S}_{m,c}$ with $r|d$ and $0<r$ in $\Omega_d$, there are at least $\frac{1}{10}2^{\omega(d)- \omega(\gcd(d,m))}$ distinct totally geodesic surfaces.
\end{corollary}
%%AR: didnt follow the proof of this
\begin{proof}
We note that the boundary circles corresponding to $\tilde{\sigma}_r \mathcal{S}_{m,c}$ are $\tilde{\sigma}_r \mathcal{C}_{m,c}$, and by Proposition \ref{prop1} one can choose at least $2^{\omega(\delta)-1}/\tau(8\gcd(d,m))$ so that none of them are $\PGL(2,\mathcal{O}_d)$ equivalent to the other. It is then automatic that none of them are $\Gamma_d$ equivalent to the other, and so corresponding geodesic surfaces in $\Omega_d$ must all be distinct. To prove the inequality, if we write $2^k\| 8\gcd(d,m)$, we first have
\[
\tau(8\gcd(d,m)) = (k+1)\prod_{p|\gcd(d,m),~p\neq 2} 2 \leq 5 \prod_{p|\gcd(d,m),~p\neq 2} 2 \leq 5 \cdot 2^{\omega(\gcd(d,m))}.
\]
We then use
\[
2^{\omega(\delta)-1} \geq 2^{\omega(d)-1}.
\]
\end{proof}
\subsection{Constructing closed embedded orientable totally geodesic surfaces}
\label{exists}
We begin with an explicit construction of totally geodesic surfaces.
\begin{proposition}
\label{thereis}
For any $r|d$, $m,c \in \mathbb{Z}$, such that $0<m<d/2$ and $0<m^2-dc<d/4$,
\[
\sigma_r^*A \sigma_r= \sigma_r^*\begin{pmatrix}d &m\sqrt{-d}\\ -m\sqrt{-d} & dc\end{pmatrix}\sigma_r
\]
corresponds to an orientable embedded totally geodesic surface in $\Omega_d$.
\end{proposition}
\begin{proof}
Let $D=-\det(A)=m^2d-d^2c$ and notice that $D= d(m^2-dc) < d(d/4)$ by assumption. Hence $d^2>4D$. As above let $Q_A$ be the binary Hermitian form corresponding to $A$. Note that since $\tilde{\sigma}_r\in B_d$
normalizes $\Gamma_d$ by construction we have $\sigma_r^{-1} \Gamma_d \sigma_r = \Gamma_d$.

It will be instructive to deal with the case of $\sigma_1$ first.  Since $\sigma_1\in \Gamma_d$, it suffices to work with $A$.  Set $\tilde{A} = \begin{pmatrix}dc &-m\sqrt{-d}\\ m\sqrt{-d} & d\end{pmatrix} = -DA^{-1}$. Then
with $\gamma=\begin{pmatrix} x&y  \\ z & w \end{pmatrix} \in \Gamma_d$ we have:

\begin{align*}
\mathrm{Tr}(\gamma^*A\gamma \tilde{A} )
=&\mathrm{Tr}\left(\begin{pmatrix}\bar{x} & \bar{z} \\ \bar{y} & \bar{w}\end{pmatrix} A \begin{pmatrix} x&y  \\ z & w \end{pmatrix} \begin{pmatrix}dc &- m\sqrt{-d}\\ m\sqrt{-d} & d\end{pmatrix} \right)\\
= & \mathrm{Tr}\left( \begin{pmatrix} Q_A(\bar{x},\bar{z}) & \begin{pmatrix}\bar{x} & \bar{z}\end{pmatrix}A\begin{pmatrix}y\\w\end{pmatrix}  \\   \begin{pmatrix}\bar{y} & \bar{w}\end{pmatrix}A\begin{pmatrix}x\\z\end{pmatrix}   & Q_A(\bar{y},\bar{w}) \end{pmatrix} \begin{pmatrix}dc &- m\sqrt{-d}\\ m\sqrt{-d} & d\end{pmatrix} \right)\\
= & dc Q_A(\bar{x},\bar{z}) + dQ_A(\bar{y},\bar{w}) + 2\mathrm{Re}\left( m\sqrt{-d}\begin{pmatrix}\bar{x} & \bar{z}\end{pmatrix}A\begin{pmatrix}y\\w\end{pmatrix} \right).
\end{align*}
By choice of $A$, we see that for $x,y \in \mathcal{O}_d$, $Q_A(x,y)$ always is a multiple of $d$, and so we have
\begin{multline*}
\mathrm{Tr}(\gamma^*A\gamma \tilde{A} ) \equiv 2\mathrm{Re}\left( m\sqrt{-d}\begin{pmatrix}\bar{x} & \bar{z}\end{pmatrix}\begin{pmatrix}d &m\sqrt{-d}\\ -m\sqrt{-d} & dc\end{pmatrix} \begin{pmatrix}y\\w\end{pmatrix} \right) \pmod{d^2}\\
=2\mathrm{Re}\left( m\sqrt{-d}(y(\bar{x}d -\bar{z}m\sqrt{-d}) +w(\bar{x}m\sqrt{-d}-\bar{z}dc) )\right) \\
\equiv -2m^2d\mathrm{Re}\left(  w\bar{x}-y\bar{z} \right) \pmod{d^2}.
\end{multline*}
Since $\det(\gamma)=1=xw-yz$ we have
\[
\mathrm{Re}\left(  w\bar{x}-y\bar{z} \right) = \mathrm{Re}\left(  w\bar{x}-y\bar{z} -wx+yz +1\right) = \mathrm{Re}\left(  w(\bar{x}-x)-y(\bar{z}-z)\right) + 1 \equiv 1 \pmod{d}.
\]
Putting these calculations together we deduce that:
\[\mathrm{Tr}(\gamma^*A\gamma \tilde{A} )\equiv -2m^2d \equiv -2(m^2d-d^2c) = -2D \pmod{d^2}\]
From the definition of $\tilde{A}$ and the fact noted above that $d^2>4D$, it follows that the associated surface is embedded by Corollary \ref{cor_embed}.

To prove orientability
of the surface associated to $A$, we apply Lemma \ref{lem:orient} with $a=d$, $B= m\sqrt{-d}$, and $c=cd$
to see that it is equivalent to the non-existence of $\gamma \in \Gamma_d$ such that
\begin{align*}
-1&=|x|^2+2\mathrm{Re}\left(\frac{\sqrt{-d}}{d}mz\bar{x}\right)+c|z|^2\\
w&=\frac{m\sqrt{-d}}{d} (z+\bar{z})-\bar{x}\\
y&=\frac{m\sqrt{-d}}{d} (x+\bar{x})-c\bar{z}.
\end{align*}
Assume
by way of contradiction that there exists such a $\gamma$. Let $x=x_1+\sqrt{-d}x_2$ and $z=z_1+\sqrt{-d}z_2$. Because $w,y \in \mathcal{O}_d$, we see from the second and the third condition that $z_1$ and $x_1$ must be divisible by $d/\gcd(m,d)$. Because $m$ satisfies $0<m^2-dc<d/4$, $m$ is not divisible by $d$, so there is a prime factor $p|d/\gcd(m,d)$. If we rearrange the first equation, we get
\[
-1=x_1^2+dx_2^2 + 2m(z_1x_2-z_2x_1) + cz_1^2 + dcz_2^2,
\]
and taking it modulo $p$ implies $-1 \equiv 0\pmod{p}$, which is a contradiction.

Now the orientability of the surfaces associated to $\sigma_r^* A \sigma_r$ (embedded or not) follows easily from above.  Briefly, since $\sigma_r$ normalizes $\Gamma_d$, so the required stabilizer of $\sigma_r^*A\sigma_r$
in $\Gamma_d$ is the conjugate $\sigma_r^{-1} F_A \sigma_r$, where $F_A$ is the Fuchsian subgroup corresponding to $A$. So the totally geodesic surfaces corresponding to $A$ and $\sigma_r^*A\sigma_r$ are isometric, and orientability follows.

We now discuss embedding the surface associated to $\sigma_r^*A\sigma_r$.  To that end, for $\gamma\in \Gamma_d$ set $\tilde{\gamma}=\sigma_r^{-1}\gamma \sigma_r$. Since $\sigma_r$ normalizes $\Gamma_d$ it suffices to consider

\[
\mathrm{Tr}\left((\tilde{\gamma})^*\sigma_r^*A\sigma_r \tilde{\gamma} \sigma_r^{-1}\tilde{A}(\sigma_r^{-1})^* \right).
\]
A calculation shows that this is simply $\mathrm{Tr}(\gamma^*A\gamma \tilde{A} )$ and embedding in this case now follows as above.\end{proof}

\subsection{Enumerating the surfaces}
We are ready to present a lower bound for the number of closed orientable embedded totally geodesic surfaces:

\begin{corollary}\label{lower}
For all sufficiently large $d$, $\Omega_d$ contains $\gg d$ closed orientable embedded totally geodesic surfaces.
\end{corollary}
\begin{proof} Let $\mathcal{C}$ be a circle of discriminant $D= d(m^2-dc)$ as in the statement of Proposition \ref{thereis}.  The Hilbert Symbol of the invariant quaternion algebra of $\Stab^+(\mathcal{C},\Gamma_d)$ is given by (\cite[Chapter 9.6]{MR2})
\[
\biggl ({\frac{-d,D}{\mathbb{Q}}}\biggr ) = \biggl ({\frac{-d,d(m^2-dc)}{\mathbb{Q}}}\biggr ) \cong  \biggl ({\frac{-d,(m^2-dc)}{\mathbb{Q}}}\biggr )
\]
If $0<m^2-dc < \frac{d}{4}$ is not of the form $n^2+dy^2$ then the corresponding surface is closed by Lemma \ref{straight_line}(2).

By Corollary \ref{duplicate}, for each fixed $m,c$ there are at least $\frac{1}{10} 2^{\omega(d)-\omega(\gcd(m,d))}$ distinct totally geodesic surfaces of the form $\tilde{\sigma}_r\mathcal{S}_{m,c}$. So the total number of closed orientable embedded totally geodesic surfaces is bounded from below by
\[
\frac{1}{10}\sum_{r\in \mathcal{D}} 2^{\omega(d)-\omega((m,d))}
\]
which is $\gg d$ by Theorem \ref{lem2}.
\end{proof}
\section{The Bianchi groups II: Upper bounds and finiteness}
We first state a general upper bound for embedded totally geodesic surfaces in $\Omega_d$ for any $d$.
\begin{theorem}\label{thm:upper}
There are at most
\[
\ll_\epsilon  d^{1+\epsilon}
\]
distinct embedded totally geodesic surfaces in $\Omega_d$.
\end{theorem}
\begin{proof}
By Lemma \ref{lem:para} and Lemma \ref{upper}, we see that $D$ is a discriminant of an embedded totally geodesic surface only if
\[
\frac{\gcd(d,D)^2 \log^4 \frac{D}{\gcd(d,D)}}{D} \gg 1.
\]
Using Lemma \ref{upper2}, we see that for each $D$, there are at most
\[
\ll \prod_{p|\gcd(d,D)} 4
\]
distinct totally geodesic surface having discriminant $D$.

Now fix $e|d$ and let $D=eD_0$ with $\gcd(D_0,d/e)=1$. Then the first condition is equivalent to saying $\frac{D_0}{\log^4 D_0}\ll e$, which implies that there are at most $\ll e \log^4 d$ distinct $D$ with $\gcd(d,D)=e$ which can be a discriminant of an embedded totally geodesic surface. Summing over, we see that the total number of distinct embedded totally geodesic surface is bounded from above by
\[
\sum_{e|d} e \log^4 d  \prod_{p|e} 4  = \log^4 d  \prod_{p|d} (4p+1) = d\log^4 d \prod_{p|d} (4+1/p)  \ll_\epsilon d^{1+\epsilon},
\]
as claimed.
\end{proof}

\subsection{Bianchi orbifolds without closed embedded totally geodesic surfaces}\label{none}
When the ideal class group of $\mathbb{Q}(\sqrt{-d})$
does not contain an element of order $4$, we can say more about embedded closed totally geodesic surfaces, because we know from \cite[Theorem 8]{Vul} that every immersed totally geodesic surfaces is of the form \eqref{rep}.
\begin{lemma}\label{lem:none}
The surface associated to $\sigma_r^{-1} \mathcal{S}_{B,c}$ is an embedded closed totally geodesic surface in $\Omega_d$, only if $D/d=|B|^2-dc \in \mathcal{D}$, where $\mathcal{D}$ is defined in \eqref{D}.
\end{lemma}
\begin{proof}
We first note that
\[
\sigma_r^{-1} \begin{pmatrix} 1 & 1 \\ 0 & 1 \end{pmatrix} \sigma_r = \begin{pmatrix} 1-uv\sqrt{-d} & -u^2s \\ -v^2r & 1+uv\sqrt{-d} \end{pmatrix} \in \Gamma_d.
\]
Hence if the surface corresponding to $\sigma_r^* A \sigma_r$ is embedded, then
\begin{multline*}
\mathrm{Tr}\left(\left(\sigma_r^{-1} \begin{pmatrix} 1 & 1 \\ 0 & 1 \end{pmatrix} \sigma_r\right)^*\sigma_r^* A \sigma_r\sigma_r^{-1} \begin{pmatrix} 1 & 1 \\ 0 & 1 \end{pmatrix} \sigma_r \left(\sigma_r^* A \sigma_r\right)^{-1}\right) \\
= \mathrm{Tr}\left(\sigma_r^*\begin{pmatrix} 1 & 0 \\ 1 & 1 \end{pmatrix}  A  \begin{pmatrix} 1 & 1 \\ 0 & 1 \end{pmatrix} A^{-1} \left(\sigma_r^*\right)^{-1}\right) = \mathrm{Tr}\left(\begin{pmatrix} 1 & 0 \\ 1 & 1 \end{pmatrix}  A  \begin{pmatrix} 1 & 1 \\ 0 & 1 \end{pmatrix} A^{-1}\right)
\end{multline*}
must have modulus greater than or equal to $2$ by Corollary \ref{cor_embed}. This is equivalent to saying that $D=d(|B|^2-dc)<\frac{d^2}{4}$, i.e., $0<b_1^2+db_2^2-dc<d/4$, where $B=b_1+\sqrt{-d}b_2$.

Observing that $b_1^2+db_2^2-dc$ is a square modulo $d$ in $(0,d/4)$, and from the assumption that $A$ corresponds to a closed surface, we see that $D/d$ must belong to $\mathcal{D}$.
\end{proof}
\begin{theorem}\label{theorem_none}
For $d \in \mathcal{E}_1$ given in the appendix, $\Omega_d$, does not contain any closed embedded totally geodesic surface.
\end{theorem}
\begin{proof}
If $d\in \mathcal{E}_1$, then the ideal class group of $\mathbb{Q}(\sqrt{-d})$
does not contain an element of order $4$, and so every immersed totally geodesic surface corresponds to a circle of the form $\sigma_r^{-1} \mathcal{S}_{B,c}$.
Then by Lemma \ref{lem:none}, $D/d$ must belong to $\mathcal{D}$, which is empty for those $d\in \mathcal{E}_1$.
\end{proof}
$\mathcal{E}$ is the list of $d < 10^9$ for which $\mathcal{D}$ is empty. So we expect for these discriminants in $\mathcal{E}$, $\Omega_d$ has no closed embedded totally geodesic surfaces. In principle, one only needs to check finitely many cases to determine whether our expectation is true. However, because this will involve a serious amount of numerical experiment, we leave this as an open question.
\begin{remark}
Yet another special case is when $d=p$ is a prime which is $3 \pmod{4}$. In this case, again by \cite{Vul}, $\mathcal{S}_{m,c}$ with $m,c \in \mathbb{Z}$, $0\leq m<p$ parameterize every immersed totally geodesic surfaces in $\Omega_p$. So by Lemma \ref{lem:none}, only possible pairs of $m,c$ are those satisfying
\[
0<m^2-pc < p/4,
\]
hence there are at most $O(p)$ embedded closed totally geodesic surfaces. Hence at least for such $d$'s, the lower bound in Corollary \ref{lower} is sharp in terms of the order of magnitude.
\end{remark}
\section{The Picard group}
\label{picard}
We now focus on the case of the Picard group $\Gamma_1$. Theorem \ref{main} implies that $\Omega_1$ does not have any closed embedded totally geodesic surfaces (which as noted in \S \ref{links} also follows from the existence of alternating link complement coverings of $\Omega_1$).  From the discussion in \S \ref{congruence} it is particularly interesting to consider embedding closed totally geodesic surfaces in quotients of $\mathbb{H}^3/\Gamma$ where $\Gamma$ is a congruence subgroup of $\Gamma_1$.
Hence it is natural to investigate the existence of embedded closed totally geodesic surfaces for quotients of $\mathbb{H}^3$ by congruence subgroups of $\Gamma_1$, and in this section, we analyze certain principal congruence subgroups
$\Gamma(I):= \Gamma_1(I)$.

\subsection{Additional facts in the case of $\Gamma_1$}
\label{circles1}

We begin with the following classification result proved in \cite{MR1}.

\begin{theorem}[\cite{MR1}]\label{classi}
Any totally geodesic surface immersed in $\Omega_1$ corresponds to a circle $\mathcal{C}$ which is $\Gamma_1$-equivalent to one of the following:
\begin{itemize}
\item $\mathcal{C}_D~:~|z|^2-D=0$,
\item $\mathcal{C}_{D,1}~:~2|z|^2+z+\bar{z}-\frac{D-1}{2}=0$ (when $D\equiv 1 \pmod{4}$),
\item $\mathcal{C}_{D,2}~:~2|z|^2+iz-i\bar{z}-\frac{D-1}{2}=0$ (when $D\equiv 1 \pmod{4}$), or
\item $\mathcal{C}_{D,3}~:~2|z|^2+(1+i)z+(1-i)\bar{z}-\frac{D-2}{2}=0$ (when $D\equiv 2 \pmod{4}$).
\end{itemize}
\end{theorem}

We will use the notation $F_D,~F_{D,1},~F_{D,2},$ and $F_{D,3}$ to denote the subgroups $\Stab^+(C_{D,*},\Gamma_1)$. For convenience, we record a special case of Lemma \ref{straight_line}(2).

\begin{lemma}
\label{sumofsquares}
$F_D,~F_{D,1},~F_{D,2},$ and $F_{D,3}$ are non-cocompact arithmetic Fuchsian groups if and only $D$ is the sum of two squares.\end{lemma}

A simple consequence of this and Theorem \ref{classi} is the following.

\begin{lemma}\label{lem12}
Suppose that $D$ is not a sum of two squares, and let $C_{D,*}$ denote one of the circles listed in Theorem \ref{classi}. Then $\begin{pmatrix}1 & 1 \\ 0 &1 \end{pmatrix}C_{D,*} \cap C_{D,*}$ is nontrivial.
\end{lemma}
\begin{proof}
Because each of these circles has radius $\sqrt{D}$ or $\sqrt{D}/2$, and because $D\geq 3$, each of the radii is $>1$.  Hence translating by $1$ gives a nontrivial intersection.
\end{proof}

\begin{remark} \label{hermitian_represent}
Note that another consequence of Theorem \ref{classi} is that, in the context of $\Gamma_1$, any primitive binary Hermitian form associated to a primitive triple $(a,B,c)$ as in \S \ref{parameter} represents $1$ or $2$.\end{remark}

\subsection{Degree one primes}
\label{split}
In this subsection we will prove to prove Theorem \ref{main2}. Thus throughout this section $\mathcal{P}$ is a degree one prime in $\mathbb{Z}[i]$, or $\mathcal{P}=<1+i>$.

\begin{lemma}\label{lem22}
Assume that $(p,2D)=1$. Then $F_D,~F_{D,1},~F_{D,2},$ and $F_{D,3}$ surject onto $\PSL(2,\mathbb{Z}[i]/\mathcal{P})$ under the homomorphism induced by reduction modulo $\mathcal{P}$.
\end{lemma}

\begin{proof} Note that although $F_{D,1}$ and $F_{D,2}$ are not $\Gamma_1$-conjugate, they are $\PGL(2,\mathcal{O}_1)$-conjugate (using the element
$\mathrm{P}\begin{pmatrix}i & 0 \\ 0 &1 \end{pmatrix})$. For our purposes it suffices to work with only $F_{D,1}$.

It will be convenient to work in $\SL(2,\mathbb{Z}[i])$ and show that the lifts of $F_D,~F_{D,1},~F_{D,2},$ and $F_{D,3}$ surject onto $\SL(2,\mathbb{Z}[i]/\mathcal{P})$. For convenience we continue to denote these by
$F_D,~F_{D,1},~F_{D,2},$ and $F_{D,3}$.

We recall some more information from \cite{MR1} about the groups  $F_D,~F_{D,1}$ and $F_{D,3}$.
From \cite[\S 6.4, \S 6.5]{MR1}, the groups $F_{D,1}$ and $F_{D_3}$ (when they are defined) can be identified with the images of the elements of norm
$1$ in orders $\mathcal{M}$ and $\mathcal{N}$ of the quaternion algebra $B_D=\biggl ({\frac{-1,D}{\mathbb{Q}}}\biggr )$.  In addition, in all cases, the group $F_D$ can be identified with the image of the elements
of norm $1$ in the order $O=\mathbb{Z}[1,i,j,ij] \subset B_D$. As noted in \cite[\S 6.4, \S 6.5]{MR1}, $O$ differs from the orders $\mathcal{M}$ and $\mathcal{N}$ only at the prime $2$.

Any prime that ramifies the algebra $B_D$ divides $2D$, and since $p$ is chosen so that $(p,2D)=1$, $B_D$ is unramified at $p$. Furthermore, we also note that locally at $p$, $O$ is maximal.
We therefore deduce that for such primes $p$, $O^1$, $\mathcal{M}^1$ and $\mathcal{N}^1$ are dense in $\SL(2,\mathbb{Z}_p)$.

Since $\mathcal{P}$ is a degree one prime, $\SL(2,\mathbb{Z}[i]_{\mathcal{P}})\cong \SL(2,\mathbb{Z}_p)$, and one checks that the image of $F_D,~F_{D,1}$ and $F_{D,3}$ in $\SL(2,\mathbb{Z}[i]_{\mathcal{P}})$ under the inclusion map
$\SL(2,\mathbb{Z}[i])\hookrightarrow \SL(2,\mathbb{Z}[i]_{\mathcal{P}})$ coincides with that of $O^1$, $\mathcal{M}^1$ and $\mathcal{N}^1$ described above. We deduce from this that
the groups $F_D,~F_{D,1}$ and $F_{D,3}$ surject $\SL(2,\mathbb{Z}[i]/\mathcal{P})\cong \SL(2,\mathbb{Z}/p\mathbb{Z})$. From the remarks above, the same holds for $F_{D,2}$ and the result now follows.
\end{proof}

We now commence with the proof of Theorem \ref{main2}. We deal with the split prime case first. It suffices to only consider the case of the prime ideal $\mathcal{P}$ since $\mathbb{H}^3/\Gamma(\mathcal{P})$ and
$\mathbb{H}^3/\Gamma(\overline{\mathcal{P}})$ differ by an orientation-reversing isometry.

Thus, by way of contradiction, assume that a circle $\mathcal{C}$ corresponds to a closed embedded totally geodesic surface $S$ in $\Gamma(\mathcal{P})$.  From Theorem \ref{classi}, $\mathcal{C}$ is $\Gamma_1$-equivalent to one of
the circles $C_{D,*}$, and so let $\gamma \in \Gamma_1$ be chosen so that
$\gamma \mathcal{C}=\mathcal{C}_{D,*}$.

Then the subgroup $G = \Stab(\mathcal{C},\Gamma_1)$ is given by $G=\gamma^{-1}F_{D,*}\gamma$. Also, by Lemma \ref{lem12}, $\begin{pmatrix}1 & 1 \\ 0 &1 \end{pmatrix} \gamma \mathcal{C} \cap \gamma \mathcal{C}$ is nontrivial, meaning that $\gamma^{-1 }\begin{pmatrix}1 & 1 \\ 0 &1 \end{pmatrix} \gamma \mathcal{C} \cap \mathcal{C}$ is nontrivial. From these, we deduce that any element in
\[
G\gamma^{-1 }\begin{pmatrix}1 & 1 \\ 0 &1 \end{pmatrix} \gamma  G=\gamma^{-1 }F_{D,*}\begin{pmatrix}1 & 1 \\ 0 &1 \end{pmatrix}F_{D,*} \gamma
\]
sends $\mathcal{C}$ to a circle that intersects $\mathcal{C}$ nontrivially.

Now consider the image of $G$ in $\PSL(2,\mathbb{Z}[i]/\mathcal{P})$. From above this coincides with the image of
$\gamma^{-1}F_{D,*}\gamma$. Hence, if $D$ is coprime to $p$, Lemma \ref{lem22} shows that the image of $G$ is $\PSL(2,\mathbb{Z}[i]/\mathcal{P})$.  Putting these remarks together, it follows that
 \[
 G \gamma^{-1 }\begin{pmatrix}1 & 1 \\ 0 &1 \end{pmatrix} \gamma  G
\]
contains an element in $\Gamma(\mathcal{P})$, which contradicts  the assumption $S$ being embedded.  Therefore $p$ must divide $D$.

Assuming this is the case, apply Lemma \ref{lem:para} with
\[
\gamma = \begin{pmatrix}1-txy & ty^2 \\ -tx^2 & 1+txy\end{pmatrix} \in \Gamma(\mathcal{P}),
\]
where $t=\pi$, a generator for the prime ideal $\mathcal{P}$. Hence,
\[
\mathrm{Tr}(\gamma^*A\gamma A^{-1}) = 2-\frac{p Q_A(-\bar{x},\bar{y})^2 }{D}.
\]
Now using Remark \ref{hermitian_represent} there exist $x,y \in \mathbb{Z}[i]$ such that
\[
 Q_A(-\bar{x},\bar{y})^2 \leq 4.
\]
For such $x,y$ we have
\[
2>\mathrm{Tr}(\gamma^*A\gamma A^{-1}) = 2-\frac{p Q_A(-\bar{x},\bar{y})^2 }{D} \geq 2-4=-2
\]
where the equality holds only if $p=D$. Now $p$ splits to $\mathbb{Q}(i)$ and so is a sum of two squares. However, because $S$ is a closed surface, this cannot happen by Lemma \ref{sumofsquares}.

Therefore we have for $x,y$ as above,
\[
\left|\mathrm{Tr}(\gamma^*A\gamma A^{-1})\right| <2,
\]
which is a contradiction to the assumption that $S$ is embedded, by Corollary \ref{cor_embed}.\\[\baselineskip]
\noindent We now handle the case of $\Gamma(1+i)$.  Perhaps the simplest way to deal with this case is topologically as follows.  The fundamental group $B$ of the complement of the Borromean rings in $S^3$ is an alternating arithmetic link complement covering $\Omega_1$ (see \cite{Ha} for example).  Moreover, \cite{FN} proves that $B$ can be realized as a normal subgroup of $\Gamma_1$ of index $24$ arising as the normal closure in $\Gamma_1$ of the set
$\{\begin{pmatrix}1 & 4\\ 0 & 1\end{pmatrix}, \begin{pmatrix}1 & 1+i\\ 0 & 1\end{pmatrix}\}$. These elements
are visibly in $\Gamma(1+i)$, and so $\Gamma(1+i)$ being a normal subgroup, contains this normal closure; i.e. $B < \Gamma(1+i)$. As discussed in \S \ref{links}, it follows from \cite{Men} that $\mathbb{H}^3/B$ cannot contain a closed embedded
totally geodesic surface, hence the same holds for $\mathbb{H}^3/\Gamma(1+i)$.\qed
\begin{remark} It is known that $\mathbb{H}^3/\Gamma(2)$ is also homeomorphic to a link complement in $S^3$ \cite{Ba0}, and also cannot contain a closed embedded totally geodesic surface (using \cite{Oe}). However, the methods of this
paper shows that $\mathbb{H}^3/\Gamma((1+i)^3)$ (homeomorphic to the complement of a link of $12$ components \cite{BGR}, \cite{BGRa}) does contain a closed embedded totally geodesic surface (of genus $3$).  This was also known to M. Stover.\end{remark}

\subsection{The case of $\Gamma_{1,0}(n)$}
\label{someothers}
In this subsection we construct closed embedded totally geodesic surfaces in $\mathbb{H}^3/\Gamma_{1,0}(n)$, for certain $n \in  \mathbb{Z}$ (some of this was also noticed by M. Stover).
For convenience, we abbreviate $\Gamma_{1,0}(n)$ to $\Gamma_0(n)$.

Recall that a result of Fermat shows that an integer $n$ expressed as $2^a\prod p^b\prod q^c$ where the products are over primes $p\equiv 1\pmod{4}$ and $q\equiv 3\pmod{4}$ is a sum of two squares (one of which may be zero) if and only if
all the exponents $c$ are even.

\begin{proposition}
Assume that $n\equiv 1 \pmod{4}$ is not a sum of two squares. Then $\mathbb{H}^3/\Gamma_0(n)$ contains a closed embedded totally geodesic surface.
\end{proposition}
\begin{proof}
We claim that the Hermitian matrix
\[
A= \begin{pmatrix}\frac{n(n-1)}{2} & n \\ n & 2\end{pmatrix}
\]
corresponds to a closed embedded totally geodesic surface $S$ in $\mathbb{H}^3/\Gamma_0(n)$. To that end, note that the discriminant of $S$ is $n$ and so Lemma \ref{sumofsquares} shows that $S$ is closed. Let
\[
\tilde{A}  = A \begin{pmatrix}n & 0 \\ 0 & n\end{pmatrix} = \begin{pmatrix}2 &- n \\ -n & \frac{n(n-1)}{2}\end{pmatrix},
\]
and compute its trace modulo $n$:
\begin{align*}
\mathrm{Tr} \left( \gamma^* A \gamma \tilde{A}\right) &\equiv \mathrm{Tr} \left( \begin{pmatrix}\bar{x} & 0 \\ \bar{y} & \bar{w}\end{pmatrix} \begin{pmatrix}0 & 0 \\ 0 & 2\end{pmatrix} \begin{pmatrix}x & y \\ 0 & w\end{pmatrix} \begin{pmatrix}2 & 0 \\ 0 & 0\end{pmatrix}\right)\pmod{n}\\
&\equiv 0 \pmod{n}.
\end{align*}
Observe that $Q_A(x,y)$ is always an even number, from which we have:
\begin{align*}
\mathrm{Tr} \left( \gamma^* A \gamma \tilde{A}\right)& = 2Q_A(\bar{x},\bar{z}) - 2\mathrm{Re}\left(n \begin{pmatrix}\bar{x} & \bar{z}\end{pmatrix} \begin{pmatrix}\frac{n(n-1)}{2} & n \\ n & 2\end{pmatrix} \begin{pmatrix}y \\ w\end{pmatrix} \right)+ \frac{n(n-1)}{2}Q_A(\bar{y},\bar{w}) \\
&\equiv -2 \mathrm{Re}( \bar{x}w + \bar{z}y) \pmod{4}\\
&\equiv 2 \mathrm{Re}( xw - zy) \pmod{4}\\
&=2.
\end{align*}
Therefore
\[
\mathrm{Tr} \left( \gamma^* A \gamma \tilde{A}\right) \equiv -2n \pmod{4n},
\]
and so
\[
\left|\mathrm{Tr} \left( \gamma^* A \gamma A^{-1}\right)\right|\geq 2.
\]
Therefore $S$ is embedded, by Corollary \ref{cor_embed}.
\end{proof}
\begin{proposition}
Assume that $n \equiv 3 \pmod{4}$ is not a sum of two squares. Then $\mathbb{H}^3/\Gamma_0(2n)$ contains a closed embedded totally geodesic surface.
\end{proposition}
\begin{proof}
We apply the same argument to
\[
A= \begin{pmatrix}n(n-1) & n \\ n & 1\end{pmatrix}.
\]
from which we infer
\[
\mathrm{Tr} \left( \gamma^* A \gamma \tilde{A}\right) \equiv -2n \pmod{4n}.
\]
Then we apply Corollary \ref{cor_embed}.
\end{proof}
We combine these two to conclude the following:
\begin{corollary}
In the following cases, $\mathbb{H}^3/\Gamma_0(n)$ contains a closed embedded totally geodesic surface
\begin{enumerate}
\item $n$ is even and $n$ has a prime factor $\equiv 3 \pmod{4}$, or
\item $n$ has two distinct prime factors $\equiv 3 \pmod{4}$.
\end{enumerate}
\end{corollary}
\begin{proof}
If $n|m$ and $\mathbb{H}^3/\Gamma_0(n)$ contains a closed embedded totally geodesic surface, then so does $\mathbb{H}^3/\Gamma_0(m)$.
\end{proof}

\newpage
\appendix

\section{Exceptional discriminants}
We define $\mathcal{E}_1$ to be the
list of $d$'s in the following table. Here $Cl_{\mathbb{Q}(\sqrt{-d})}$ denote the ideal class group of $\mathbb{Q}(\sqrt{-d})$ and $h_{-d}$ the class number. We only indicate the structure of the ideal class group when $h_{-d}$
is divisible by $4$.

\[
\begin{array}{ | c | c | c | c | c | c | c | }
\hline
	d & h_{-d} & Cl_{\mathbb{Q}(\sqrt{-d})} &&	d & h_{-d} & Cl_{\mathbb{Q}(\sqrt{-d})}\\ \hline
	1& 1& &&	215& 14&   \\
	2& 1& &&	231& 12& \mathbb{Z}_6\times \mathbb{Z}_2  \\
	3& 1& &&	239& 15&  \\
	5& 2&  &&	255& 12& \mathbb{Z}_6\times \mathbb{Z}_2 \\
	6& 2& &&	287& 14&   \\
	7& 1&  &&	311& 19&  \\
	10& 2&  &&	335& 18&  \\
	11& 1& &&	359& 19&  \\
	15& 2& &&	455& 20& \mathbb{Z}_{10}\times \mathbb{Z}_2 \\
	19& 1& &&	479& 25& \\
	21& 4& \mathbb{Z}_2\times \mathbb{Z}_2&&	519& 18&  \\
	23& 3&  &&	551& 26&  \\
	26& 6& &&	591& 22&  \\
	30& 4& \mathbb{Z}_2\times \mathbb{Z}_2&&	615& 20& \mathbb{Z}_{10}\times \mathbb{Z}_2 \\
	31& 3&  &&	671& 30& \\
	33& 4& \mathbb{Z}_2\times \mathbb{Z}_2 &&	719& 31&\\
	35& 2& &&	815& 30&  \\
	42& 4& \mathbb{Z}_2\times \mathbb{Z}_2 &&	935& 28& \mathbb{Z}_{14}\times \mathbb{Z}_2 \\
	47& 5& &&	951& 26&  \\
	51& 2& &&	1095& 28& \mathbb{Z}_{14}\times \mathbb{Z}_2  \\
	59& 3& &&	1335& 28& \mathbb{Z}_{14}\times \mathbb{Z}_2  \\
	71& 7& &&	1455& 28& \mathbb{Z}_{14}\times \mathbb{Z}_2 \\
	79& 5& &&	1599& 36& \mathbb{Z}_{18}\times \mathbb{Z}_2  \\
	87& 6& &&	2015& 52& \mathbb{Z}_{26}\times \mathbb{Z}_2 \\
	105& 8& \mathbb{Z}_2\times \mathbb{Z}_2\times \mathbb{Z}_2&&	2415& 40& \mathbb{Z}_{10}\times \mathbb{Z}_2\times \mathbb{Z}_2 \\
	119& 10& &&	2679& 52& \mathbb{Z}_{26}\times \mathbb{Z}_2  \\
	131& 5&&&	3135& 40& \mathbb{Z}_{10}\times \mathbb{Z}_2\times \mathbb{Z}_2  \\
	143& 10& &&	3255& 40& \mathbb{Z}_{10}\times \mathbb{Z}_2\times \mathbb{Z}_2  \\
	159& 10& &&	3471& 60& \mathbb{Z}_{30}\times \mathbb{Z}_2  \\
	167& 11& &&	8151& 88& \mathbb{Z}_{22}\times \mathbb{Z}_2\times \mathbb{Z}_2  \\
	191& 13& &&	10335& 88& \mathbb{Z}_{22}\times \mathbb{Z}_2\times \mathbb{Z}_2  \\
	195& 4& \mathbb{Z}_2\times \mathbb{Z}_2 &&&&\\
\hline
\end{array}
\]
and define $\mathcal{E}_2$ to be
\begin{multline*}
\{14,~17,~34,~39,~41,~55,~66,~69,~95,~111,~183,~285,~327,~399,~471,\\
~759,~791,~831,~959,~1055,~1119,~1239,~1551,~3311,~4199,~15015\}.
\end{multline*}
Let $\mathcal{E}=\mathcal{E}_1 \cup \mathcal{E}_2$.
\begin{proposition}
If $\Omega_d$ has no embedded closed totally geodesic surfaces, then $d \in \mathcal{E}$ or $d>10^9$.
\end{proposition}
\begin{proof}
This is checked using Mathematica, and the code can be shared upon request.
\end{proof}
We conjecture that $\mathcal{E}$ is the complete list of $d$ for which $\Gamma_d \backslash \mathbb{H}^3$ has no embedded closed totally geodesic surfaces.
\end{document}